\documentclass{amsart}
\usepackage{amsaddr}

\usepackage{geometry}
\usepackage[english]{babel}

\usepackage{latexsym,amsfonts}
\usepackage{amssymb}
\usepackage{verbatim}
\usepackage{mathrsfs}

\def\tr{\mathrm{tr}}
\def\C{\mathscr{C}}
\def\Cl{\mathcal{C}}
\def\ClS{\mathcal{S}}

\def\diag{\mathrm{diag}}
\def\GL{\mathrm{GL}}    
\def\Mat{\mathrm{Mat}}   
\def\I{\mathrm{I}}
\def\id{\mathsf{id}}

\def\PSU{\mathrm{PSU}}
\def\PSL{\mathrm{PSL}}

\def\PSp{\mathrm{PSp}}
\def\Or{\mathrm{O}}
\def\SL{\mathrm{SL}}  
\def\SU{\mathrm{SU}}
\def\Sp{\mathrm{Sp}}
\def\SO{\mathrm{SO}}

\def\Sym{\mathrm{Sym}}
\def\Alt{\mathrm{Alt}}
\def\wSL{\widehat{\SL}}
\def\wGL{\widehat{\GL}}
\def\P{{\rm P}}
\def\N{\mathcal{N}}
\def\V{\mathcal{V}}
\def\Vb{\mathbf{V}}
\def\K{\mathcal{K}}

\def\equad{\quad \textrm{ and } \quad}

\newcommand{\Z}{\mathbb{Z}}   
    
\newcommand{\F}{\mathbb{F}}   

\newtheorem{theorem}{Theorem}[section] 
\newtheorem{lemma}[theorem]{Lemma}     
\newtheorem{corollary}[theorem]{Corollary}
\newtheorem{proposition}[theorem]{Proposition}

\theoremstyle{definition}
\newtheorem{defin}[theorem]{Definition}

\numberwithin{equation}{section}

\begin{document}

\title{On the $(2,3)$-generation of the finite symplectic groups}

\author{M.A. Pellegrini and M.C. Tamburini Bellani}
\address{Dipartimento di Matematica e Fisica, Universit\`a Cattolica del Sacro Cuore,
Via Musei 41, 25121 Brescia, Italy}

\email{marcoantonio.pellegrini@unicatt.it}
\email{mariaclara.tamburini@gmail.com}

\begin{abstract}
This paper is a new important step towards the complete classification of the
finite simple groups which are $(2, 3)$-generated.
In fact, we prove that the  symplectic groups $\Sp_{2n}(q)$ are 
$(2,3)$-generated for all $n\geq 4$. 
Because of the existing literature, this result implies that the 
groups $\PSp_{2n}(q)$ are $(2,3)$-generated for all $n\geq 2$, with the exception of 
$\PSp_4(2^f)$ and $\PSp_4(3^f)$.
\end{abstract}


\maketitle

\section{Introduction}

A group is said to be $(2,3)$-generated if it can be generated by an involution and an element of order
$3$. By a famous result of Liebeck and Shalev \cite{LS}, the finite classical simple groups are
$(2,3)$-generated, apart from the two infinite families $\PSp_4(q)$ with $q=2^f,3^f$, and a finite list 
$\mathcal{L}$ of exceptions.
This list $\mathcal{L}$ includes the ten groups $\PSL_2(9), \PSL_3(4),\PSL_4(2),\PSU_3(9),$
$\PSU_3(25), \PSU_4(4)\cong \PSp_4(3),\PSU_4(9),\PSU_5(4),\P\Omega_8^+(2), \P\Omega_8^ +(3)$  (see \cite{Ischia,35,Max}).
It was also proved that $\mathcal{L}$ contains no other linear group $\PSL_n(q)$, no other unitary group 
$\PSU_n(q^2)$, no other symplectic group $\PSp_4(q)$, no symplectic group $\PSp_6(q)$, and no orthogonal group $\Omega_7(q)$
(see \cite{CDM,Sp6,SL12,Unit,TV}).
However, the problems of determining whether other exceptions exist and  finding $(2,3)$-generating pairs 
for the positive cases are still open (see, for instance, \cite[Problem 18.98]{KM}).

In this paper we show that $\mathcal{L}$ contains no symplectic group 
$\PSp_{2n}(q)$ for $n\geq 4$, giving a constructive proof of  the following result.

\begin{theorem}\label{main}
 The finite symplectic groups  $\Sp_{2n}(q)$  are $(2,3)$-generated for all $n\geq 4$.
\end{theorem}

We recall that the $(2,3)$-generation of finite symplectic groups has been studied for very small and for large 
dimensions.
More precisely, the groups $\PSp_{2n}(q)$ are $(2,3)$-generated if $n=2$ and $\gcd(q,6) = 1$ (see \cite{CDM}), 
if $n=3$ (see \cite{Sp6,TV}), if  $n\geq 25$ and $q$ is odd or if   $n\geq 40$ and $q$ is even (see \cite{ST}).
On the other hand, we point out that the groups $\Sp_4(q)$ and the groups $\Sp_6(q)$ with $q$ odd are not 
$(2,3)$-generated, see \cite{Sp4,V}.
In \cite{Vse2} Vasilyev and  Vsemirnov 
have shown that certain  matrices, of order $2$ and $3$, generate $\Sp_8(\Z)$.
The same authors have exhibited in \cite{Sp10} a pair of $(2,3)$-generators for $\Sp_{10}(\Z)$.
It follows that, for all primes $p$, the images modulo $p$ of these matrices generate $\Sp_8(\F_p)$ and $\Sp_{10}(\F_p)$,
respectively.

In view of these results we obtain the following classification.

\begin{corollary}\label{mainsym}
Let $n\geq 2$. The groups $\Sp_{2n}(q)$ are $(2,3)$-generated, except when $n=2$ and when $n=3$ and $q$ is odd.
The  groups $\PSp_{2n}(q)$ are $(2,3)$-generated  with the exception of 
$\PSp_4(2^f)$ and $\PSp_4(3^f)$, $f\geq 1$.
\end{corollary}

Theorem \ref{main} will be proved exhibiting pairs $x,y$ of $(2,3)$-generators.
In particular, the generators for $n\neq 5,6,8$ are provided in Section \ref{sec2}, 
where we give also some preliminary results.
A crucial role in our analysis will be played by the commutator $[x,y]=xy^{-1}xy$.
For this reason we describe its action in Section \ref{azione}.
The remaining parts of this paper are devoted to the proof of Theorem \ref{main}.
In fact, this result follows from 
Theorems \ref{main4} and \ref{main5} when $n\in \{4,5\}$;
Proposition  \ref{q=2} when $q=2$ and $n\geq 6$; Corollaries \ref{p210} and \ref{pno213}  when  $q>2$ and $n=10$ or $n\geq 12$;
Propositions   \ref{main7},  \ref{main9}, \ref{main11} when $q>2$ and $n \in \{7,9,11\}$;
Propositions \ref{main6} and  \ref{main8} when $q>2$ and $n \in \{6,8\}$.

Because of the existing literature and Corollary \ref{mainsym}, the problem of the $(2,3)$-generation of
the finite simple groups remains open 
only for the following orthogonal groups:
\begin{itemize}
\item $\Omega_{2n+1}(q)$, where  $4\leq n\leq 39$ and $q$ is odd;
\item $\P\Omega_{2n}^+(q)$, where  $4\leq n\leq 24$, any $q$;
\item $\P\Omega_{2n}^-(q)$, where  $4\leq n\leq 43$ if $q$ is odd and  $n\geq 4$ if $q$ is even.
\end{itemize}

\section{The generators for $n\neq 5$  and preliminary results}\label{sec2}

Let $q=p^f$ be a power of the prime $p$. 
Denote by $\F$ the algebraic closure of $\F_p$ and by 
$\C=\{e_1,\ldots,e_{n}, e_{-1},\ldots,e_{-n} \}$ the canonical basis of $\V=\F^{2n}$.
Make $\GL_{2n}(\F)$  act on the left on $\V$ and call
$\Sp_{2n}(\F)$ the group of isometries of the Gram matrix
$$J=\begin{pmatrix} 0  & - \I_n \\ \I_n & 0\end{pmatrix}.$$
Setting   $V=\sum\limits_{i=1}^n \F_q e_i$ we have $e_{-i}=Je_i$
and $\Vb:=V \oplus JV=\F_q^{2n}$.

The stabilizer of $V$ in $\Sp_{2n}(\F_q)$ is the maximal subgroup $N\rtimes\wGL_n(\F_q)$, where
$$N= \left\{
\begin{pmatrix} \I_n & X\\ 0 & \I_{n}\end{pmatrix}\mid  X=X^T\right\}\equad 
\wGL_n(\F_q)= \left\{ \begin{pmatrix} A & 0\\  0 & A^{-T}\end{pmatrix} \mid A\in \GL_{n}(\F_q)\right\}.$$
Set $S_0=\{0\}$  and, for $1\leq \ell\leq n$, define
$$\begin{array}{rcl}
S_\ell & =&  \left\langle e_i\mid n-\ell+1\leq i\leq  n\right\rangle,\\
\wSL_\ell(\F_q) & = &\left\{\diag(\I_{n-\ell}, A, \I_{n-\ell}, A^{-T}) \mid A\in \SL_{\ell}(\F_q)\right\}.
\end{array}$$
Clearly $\wSL_\ell(\F_q)$ induces $\SL_\ell(\F_q)$ on $S_\ell$.
Moreover $\left\langle \wSL_n(\F_q), N, N^T\right\rangle=\Sp_{2n}(q)$, e.g., see Steinberg's generators in \cite[Section 11.2]{Ca}.

For $n=4$ and for $n\geq 6$  write 
$$n=3m+r, \quad r=0,1,2.$$
For any $a \in \F_q^*$ such that $\F_q=\F_p[a]$, consider the matrices  $x_1$, $x_2=x_2(a)$ of order $1$ or $2$  and the matrices $y_1$, $y_2$ of order $1$ or $3$ so defined.
\begin{itemize}
  \item[($x_1$)] The matrix $x_1\in \SL_{2n}(q)$ is obtained as follows:
\begin{itemize}
\item[$\bullet$] if $r=0$, then $x_1$ fixes $e_{\pm j}$ for all $j\geq 3$. Also,
$x_1 e_1=e_{-2}$,  $x_1 e_{-2}=e_1$, 
$x_1e_2= - e_{-1}$ and  $x_1e_{-1} = -e_2$
when $p$ is odd; 
$x_1 e_1 = e_1+e_{-1}$, $x_1 e_{-1}=e_{-1}$, $x_1 e_2=e_{-2}$ and $x_1e_{-2}=e_2$ when $p=2$;
\item[$\bullet$]  if $r>0$, then $x_1=\I_{2n}$.
\end{itemize}
\item[($x_2$)] The matrix $x_2 \in \wGL_n(q)$ fixes $e_{\pm(3j+5+r)}$ for any $j=0,\ldots,m-4$. Moreover:
\begin{itemize}
\item[$\bullet$] if $r=0$ then $x_2$ fixes $e_{\pm 1}$ and $e_{\pm 2}$; 
\item[$\bullet$] if $r=1$ then $x_2e_{\pm 1}=e_{\pm 2}$ and, if $n\geq 7$, $x_2 e_{\pm 3}=e_{\pm 3}$;
\item[$\bullet$] if $r=2$ then $x_2e_{\pm 1}=e_{\pm 4}$ and $x_2 e_{\pm 2}= e_{\pm 3}$;
\item[$\bullet$] if $n\geq 9$ and $n\neq 11$, then $x_2 e_{\pm (n-4)}=-e_{\pm (n-4)}$;
if $n=11$, then $x_2 e_{\pm (n-4)}=e_{\pm (n-4)}$;
\item[$\bullet$] $x_2 e_{\pm(3j+3+r)} = e_{\pm(3j+4+r)}$ for any $j=0,\ldots,m-2$;
\item[$\bullet$] if $n\geq 6$ and $p$ is odd, then $x_2$ acts on $\langle e_{n-1},e_{n}\rangle$ with matrix
\begin{equation}\label{matGamma}
\gamma= \begin{pmatrix}
-1 & 0 \\ a & 1
\end{pmatrix}, 
\end{equation}
and acts on $\langle  e_{-(n-1)}, e_{-n}\rangle$ with matrix $\gamma^T$;
otherwise,  $x_2$ acts on $\langle e_{n-1},e_{n}\rangle$ and on 
$\langle  e_{-(n-1)}, e_{-n}\rangle$ with respective matrices $\gamma^T$ and 
$\gamma$.
\end{itemize}
\item[($y_1$)] The matrix $y_1 \in \SL_{2n}(q)$  fixes $e_{\pm j}$ for all $j>r$ (so $y_1=\I_{2n}$ if $r=0$) and
\begin{itemize}
\item[$\bullet$] if $r\geq 1$, then $y_1 e_1= e_{-1}$ and $y_1 e_{-1} =-(e_1+e_{-1})$;
\item[$\bullet$] if $r=2$, then  $y_1 e_{\pm 2}=e_{\pm 2}$.
\end{itemize}
\item[($y_2$)] The matrix $y_2\in \wSL_n(q)$ fixes $e_{\pm j}$ for all $j\leq r$ and
\begin{itemize}
\item[$\bullet$]  $y_2$ acts on $\langle e_{\pm(r+1)},\ldots,e_{\pm (n-3)}\rangle$ as the permutation
$$\prod_{j=0}^{m-2} \left(e_{\pm(3j+1+r)},e_{\pm(3j+2+r)},e_{\pm(3j+3+r)}\right);$$
\item[$\bullet$] if $n\neq 4,8$, then $y_2$ acts on $\langle e_{n-2},e_{n-1},e_{n}\rangle$ 
respectively with 
matrices
$$\eta_1=\begin{pmatrix} 0 & 0 & 1 \\ 1 & 0 & 0 \\ 0 & 1 & 0 \end{pmatrix} \; \textrm{ if }  p >2,\quad 
\eta_2=\begin{pmatrix} 1 & 1 & 1 \\ 0 & 1 & 0 \\ 1 & 0 & 0 \end{pmatrix}\; \textrm{ if }  p=2 < q, $$
 $$ \eta_3=\begin{pmatrix} 1 & 1 & 0 \\ 1 & 0 & 0 \\ 1 & 1 & 1 \end{pmatrix}\; \textrm{ if }  q=2,$$ 
and acts on $\langle e_{-(n-2)},e_{-(n-1)},e_{-n}\rangle$ with matrix $\eta_i^{-T}$; 
otherwise, $y_2$ acts on $\langle e_{n-2}, e_{n-1},e_{n}\rangle$ and on $\langle e_{-(n-2)},e_{-(n-1)},e_{-n}\rangle$  
with respective matrices  $\eta_1$ and $\eta_1^{-T}$.
\end{itemize}
\end{itemize}
Note that  $x_1x_2=x_2x_1$, $y_1y_2=y_2y_1$ and $x_i^TJx_i=y_i^TJy_i=J$ for $i=1,2$,
whence $x_1,x_2,y_1,y_2 \in \Sp_{2n}(q)$. Thus, 
$$x:=x_1x_2, \quad y:=y_1y_2$$
have respective orders $2$ and $3$, and
$$H:=\langle x,y\rangle $$
is a subgroup of $\Sp_{2n}(q)$. Our aim is to find suitable conditions on $a\in \F_q^*$ such that $H=\Sp_{2n}(q)$. 
For $n=5$ we use the generators given in Figure \ref{gen5}.
\smallskip

An element $g\in \SL_m(q)$ is called a \emph{pseudoreflection} or a \emph{bireflection}
according as its fixed point space has respectively dimension $m-1$
or $m-2$. Transvections are a special case of pseudoreflections.
We call \emph{subgroups of root type} of $\SL_n(q)$ the conjugates of $\I_n+\F_qE_{1,2}$ under $\GL_n(q)$.

\begin{lemma}\label{McLaughlin}
Let $K$ be an irreducible subgroup of $\SL_{2m+1}(q)$ generated
by transvections. Set $\F_q^{2m+1}=U\oplus W$, where  $U$ has dimension  $2$.
Suppose that $K$ contains a subgroup $B$, generated by transvections with center in $U$
and axis containing $W$. Then $B$ fixes $U$ and, if $B_{|U}=\SL_2(q)$, then $K=\SL_{2m+1}(q)$.  
\end{lemma}

\begin{proof}
Let $L$ be the normal closure of $B$ under $K$. If $L$ is reducible,
by Clifford's Theorem $K$ preserves a decomposition
with $\ell$ homogeneous components $T_j$, such that $LT_j=T_j$, $1\leq j\leq \ell$.
The assumption $B_{|U}= \SL_2(q)$ excludes $\ell=2m+1$. 
There exists a transvection $k\in K$ such that $k T_1\neq T_1$.
By Grassmann's formula, $k$ fixes a non zero vector in $T_1$, whence the contradiction $kT_1=T_1$.
Thus $L$ is irreducible. Again by Clifford's Theorem, $L$ has no normal unipotent subgroup.
Noting that $B$ is generated by subgroups of root type,
our claim follows from  \cite[Theorem, page 364]{Mc} if $q>2$ and from \cite{Mc2} 
if $q=2$.
\end{proof}

\begin{theorem} \label{Guralnick-Saxl}
Let $n\geq 7$. Suppose that $H$ contains a subgroup $K_n$ of $\wSL_n(q)N$ which induces $\SL_n(q)$ on $V$. Then  $H=\Sp_{2n}(q)$.
\end{theorem}

\begin{proof} 
For sake of simplicity let us write $V$ for $\sum_{i=1}^n \F e_i$. Thus $V\oplus JV=\F^{2n}$.
We may identify $N$ with the space  $\mathscr{S}$ of $n\times n$ symmetric matrices over $\F$.
Then $K_n$ acts by conjugation on $N$ as $\SL_n(q)$ acts on $\mathscr{S}$ via $X\mapsto A X A^T$.
This action is irreducible, i.e., $N$ is a minimal normal subgroup of $K_n$.

So, if $K_n\cap N\neq \{1\}$, then $N\leq K_n=\wSL_n(q)N$. On the other hand, if 
$K_n\cap N= \{1\}$, then $K_n=\wSL_n(q)^g$ for some $g\in N$. Indeed,
$H^1(\SL_n(q),\mathscr{S})=0$ (see \cite[Table 4.5]{CPS},  \cite[Table page 324]{JP} and \cite[Lemma 2.15]{GS}).
Clearly, to our purposes, we may reduce to the case $K_n=\wSL_n(q)^g$.
From $g \in N$ it follows that $W=g^{-1}JV$ is a complement of $V$ in $\F^{2n}$.
We now show that the only proper subspaces of $\F^{2n}$ fixed by $K_n$ are $V$ and $W$.

It is enough to assume $g=1$. Call $U$ a non-zero $K_n$-invariant subspace of $\F^{2n}$.
If $U\cap V\neq \{0\}$, then $V\leq U$. If $V\neq U$, then $W\cap U\neq \{0\}$, whence
$W\leq U=\F^{2n}$. Similarly, if $U\cap W\neq \{0\}$, we get $U=W$ or $U=\F^{2n}$.
We are left to consider the case  $u=u_1+u_2\in U$ with $0\neq u_1\in V$ and $0\neq u_2\in W$. 
In this case $U$ projects onto $V$ and onto $W$. If the kernel of the projection
onto $W$ is $V$, then $U=\F^{2n}$. Otherwise this kernel is $\{0\}$. Since the projections are linear, 
there exists $M\in \GL_n(\F)$
such that $U=\left\{\begin{pmatrix} v\\ Mv\end{pmatrix}\mid v\in \F^n\right\}$. 
For all  $k=\diag(A_k,A_k^{-T})\in K_n$ and all $v,w\in \F^n$ we have 
$k\begin{pmatrix} v\\ w\end{pmatrix}=\begin{pmatrix} A_k v\\ A_k^{-T}w\end{pmatrix}$.
So, in order that $U$ is $K_n$-invariant, we must have $A_k^{-T}Mv =M A_kv$ for all $v$. 
It follows that the fixed matrix  $M$ coincides with $A_k^T M A_k$ for all $A_k\in \SL_n(q)$, a contradiction. 

From the irreducibility of $H$ it follows that also the normal closure $K_n^H$ of $K_n$ under $H$ is irreducible. 
If not, applying Clifford's Theorem to the normal subgroup $K_n^H$,  the group $H$ must permute $V$ and $W$.
In particular $yV=V$ and $yW=W$.  It follows 
$r=0$  and $W=xV$, which is in contradiction with $V\cap xV\neq \{0\}$.

Now, since $\wSL_n(q)$ is  generated by bireflections, we can apply \cite[Theorem 10.2]{GS} to $K_n^H$, 
obtaining the following possibilities. \\[1pt]
\noindent (a) $K_n^H=\Sp_{2n}(q_0)$ acting on its natural module $\F_{q_0}^n$, for some $q_0\leq q$.
Since $\tr(xy)=a$ if $p>2$ and $\tr(xy)=a+1$ if $p=2<q$, from $\F_p[a]=\F_q$ we get $q_0=q$, whence $H=\Sp_{2n}(q)$.\\[1pt]
(b) $V+W=V_1\oplus \ldots \oplus V_n$, where each $V_i$ has dimension $2$, and $K_n^H$ permutes the subspaces $V_i$
as $\Sym(n)$. This can easily be excluded by order reasons. Or else, one may consider that
neither $\SL_n(q)$ nor $\PSL_n(q)$ can embed into $\Sym(n)$, which has a non-trivial representation
of degree $n-1$ over $\F_q$. \\[1pt]
(c) $q$ is even and $K^H$ is an alternating or symmetric group, acting on its deleted permutation module.
This case can easily be excluded by order reasons.\\[1pt]
(d) $q$ is even and  $K_n^H=\Omega_{2n}^\pm(q_0)$ for some $q_0\leq q$. 
Let $Q : \F^{2n}\to \F$ be a non-degenerate quadratic form fixed
by $K_n^H$.  The Gram matrix, with respect to the canonical basis, of the corresponding symmetric 
bilinear form  must be $J$, by the absolute irreducibility of $K_n^H$. 
For any linearly independent $v_1,v\in V$ 
we get $f(v_1)=f(v)$ by the transitivity of $\SL_n(q)$ on $V\setminus \{0\}$. 
As $V$ is totally isotropic, $Q(v_1)= Q(v_1+v)=Q(v_1)+Q(v)$, whence $Q(v)=0$ for all
$v\in V$.  Thus, for all $u\in \F^{2n}$, we get  $f(u)=\pi_1(u)^T\pi_2(u)$, where $\pi_1$ and $\pi_2$
denote the respective  projections on $V$ and $JV$. 
If $r=0$, from $xe_1= e_1+ e_{-1}$ we get the contradiction $0=Q(e_1)= Q(e_1+e_{-1})=1$.
Otherwise, from $y e_{-1}=e_1+e_{-1}$ we get the same.
\end{proof}

In what follows it is convenient to depict the last $n-r$ basis vectors  of $V$ organized in $m$ ``rows'' of $3$ 
vectors each.
So, row $m$ consists of $e_{n-2}, e_{n-1},e_{n}$. 
Each row is $y_2$-invariant. Moreover,  $x_2$ swaps the third point of each row with the first point
of the next one.

\begin{lemma}\label{extension}
Suppose $n\geq 6$ and let $K$ be a subgroup of the stabilizer of $V$ in $\Sp_{2n}(q)$
which  fixes $S_\ell$, for some $\ell\geq 4$, inducing $\{\id\}$  on $V/S_{\ell}$ and $\SL_{\ell}(q)$ on  $S_{\ell}$.
Take $g\in \Sp_{2n}(q)$ satisfying the following conditions:
\begin{itemize}
\item[\rm{(i)}] $g^{-1} e_{n-\ell+1}= e_{n-\ell}$ and, either
\item[\rm{(ii)}] $g S_{\ell+1} = S_{\ell+1}$, or
\item[\rm{(iii)}] $g  S_{\ell-1}=S_{\ell -1}$ and $ge_{n-\ell+1}=e_{n-\ell-1}$.
\end{itemize}
Then $\langle K, K^g\rangle$ fixes $S_{\ell +1}$, inducing $\SL_{\ell +1}(q)$ on  $S_{\ell +1}$. 
\end{lemma}

\begin{proof} 
Take $k\in K$ and $v\in S_m$ with $m\geq \ell$.
We have $kv-v\in S_\ell\leq S_{m}$. Thus $K$ fixes $S_{m}$ for all $m\geq \ell$.
In particular $KS_{\ell +1}=S_{\ell +1}$. We claim that $\langle K, K^g\rangle$  fixes $S_{\ell+1}$.
This is clear in case (ii). In case  (iii), from $S_\ell=\langle e_{n-\ell+1},S_{\ell-1} \rangle$
if follows  $g^{-1}S_\ell=\langle e_{n-\ell}, S_{\ell -1} \rangle\leq S_{\ell+1}$. 
For each $k\in K$ there exists $u=u_k \in S_{\ell}$ such that
$$(g^{-1} k g)  e_{n-\ell +1} = g^{-1} k e_{n-\ell-1} = g^{-1}(e_{n-\ell-1}+u)=
e_{n-\ell+1} + g^{-1}u \in S_{\ell +1}.$$
Thus, we can set $\Gamma_1=K_{|S_{\ell+1}}$, $\Gamma_2=(K^g)_{|S_{\ell+1}}$ and $\Gamma=\langle \Gamma_1, \Gamma_2\rangle \leq \SL_{\ell+1}(q)$.

With respect to the decomposition $S_{\ell+1}=\langle e_{n-\ell}\rangle \oplus S_\ell$, 
we have $\Gamma_1\leq AB$, where
$$ A=\begin{pmatrix}
1 & 0\\
0 & \SL_\ell(q)
\end{pmatrix},\quad  B=\left\{\begin{pmatrix}
1 & 0\\
v & \I_\ell
\end{pmatrix}\mid v\in \F_q^\ell\right\} \leq \SL_{\ell+1}(q).$$ 
Note that $B$ is a minimal normal subgroup of $AB$, as the conjugation action of $AB$ on $B$ coincides with the natural action of
$\SL_\ell(q)$ on $\F_q^\ell$, which is transitive on the non-zero vectors. 
Thus, if $\Gamma_1 \cap B\neq \left\{\I_{\ell+1}\right\}$ we have $B\leq \Gamma_1=AB$.
On the other hand, if $\Gamma_1 \cap B=\left\{\I_{\ell+1}\right\}$, there exists $b^{-1} \in B$ such
that $(\Gamma_1)^b=A$, by the vanishing of the first cohomology group $H^1\left(\SL_\ell(q), \F_q^\ell\right)$.
In both cases, $\Gamma_1$ is generated by subgroups of root type. Moreover, 
the only subspace fixed by $\Gamma_1$, if any, must have dimension $1$ and be generated by $e_{n-\ell}+s$ for some $s \in S_{\ell}$.
Now, let $U$ be a subspace of $S_{\ell+1}$ fixed by $\Gamma$. If $\dim(U)\geq 2$, then $U\cap S_{\ell}\neq \{0\}$
and $U\cap g^{-1}S_{\ell}\neq \{0\}$ by Grassmann's formula.
If follows that $S_{\ell}\leq U$ and $g^{-1}S_{\ell}\cap S_{\ell+1}\leq U$, whence $U=S_{\ell+1}$.
If $\dim(U)=1$, then $U$ is generated by $e_{n-\ell}+s$, a contradiction as this space is not fixed by $\Gamma_2$.
It follows that $\Gamma$ is irreducible.

Call $L$ the normal closure of $\Gamma_1$ in $\Gamma$.
If $L$ is reducible, by Clifford's Theorem there exists a decomposition $S_{\ell+1}=\sum\limits_{i=1}^k W_i$
preserved by $\Gamma$, for some $k>1$, where  each homogeneous component $W_i$ is fixed by  $\Gamma_1$.
If $\dim(W_1)\geq 2$, then $W_1\cap S_\ell \neq \{0\}$, whence the contradiction $S_\ell\leq W_1$.
Then $\dim(W_1)=1$, with implies that  $\Gamma$ is abelian, a contradiction. Thus $L$ is irreducible.
Moreover, $L$ cannot be contained
in $\Sp_{\ell+1}(q)$ as $S_\ell$ should be a totally isotropic subspace.
We conclude $\Gamma=\SL_{\ell+1}(q)$ by  \cite[Theorem, page 364]{Mc} if $q\neq 2$, 
by \cite{Mc2} if $q=2$.
\end{proof}

To prove the next result, it is useful to introduce the following.

\begin{defin}\label{Kell}
For each $\ell\leq n$ we call $\K_\ell$ any subgroup of $H$ with the following properties: 
\begin{itemize}
\item[$(\mathrm{a}_\ell)$] $\K_\ell$ fixes $S_\ell$ inducing $\SL_\ell(q)$;
\item[\rm{($b_\ell$)}] $\K_\ell$ induces the identity on $\frac{W_\ell}{S_\ell}$, where 
$W_\ell=V+\left\langle e_{-i}\mid 1\leq i\leq n-\ell\right\rangle$.
\end{itemize}
\end{defin}
Note that  $\K_\ell$ fixes $S_m$ for all $\ell+1\leq m\leq n$. In particular $\K_\ell$ fixes $S_n=V$.

\begin{theorem} \label{wSLell}
Let $n\geq 7$ and assume that $H$ contains a subgroup $\K_b$ for some $4\leq b \leq n$.
Then $H=\Sp_{2n}(q)$.
\end{theorem}

\begin{proof}
Let $\ell\geq 4$ be the largest integer such that $H$ contains a subgroup $\K_\ell$. 
If we can show that $\ell =n$, our claim  follows from Theorem \ref{Guralnick-Saxl}. 
So, suppose  $\ell < n$.
Note that if $y^{-1} S_\ell=S_\ell$ then $xe_{n-\ell+1}= e_{n-\ell}$, except when $(r,\ell)=(2,n-1)$.
Otherwise  $y^{-1} e_{n-\ell+1}=e_{n-\ell}$.
Taking $g=x$ in the first case, 
$g=y$ in the second, by Lemma \ref{extension} the subgroup $\langle \K_\ell,\K_\ell^g\rangle$ satisfies $(\mathrm{a}_{\ell+1})$,
except when $(r,\ell)=(2,n-2)$.
Assuming first $(r,\ell)\not \in \{ (2,n-2),(2,n-1)\}$, we show that $\langle \K_\ell,\K_\ell^g\rangle$ satisfies
also $(\mathrm{b}_{\ell+1})$, proving that it is a subgroup $\K_{\ell+1}$, in contrast with the maximality of $\ell$.
We will deal with these two exceptional cases at the end.

For all $v \in R=\langle e_j,e_{-j}\mid  1\leq j\leq n-\ell-1\rangle$, there exists
$w \in R$ such that $gv=w$. 
Using assumption $(\mathrm{b}_\ell)$, for any $k\in \K_\ell$  we obtain 
$$k^g e_{\pm j} -e_{\pm j} = g^{-1} k e_{\pm h} - g^{-1} e_{\pm h}= g^{-1} (ke_{\pm h}-e_{\pm h})\in g^{-1} S_\ell\leq S_{\ell+1},$$
proving $(\mathrm{b}_{\ell+1})$ for $\langle \K_\ell,\K_\ell^g\rangle$.

We now consider the two exceptional cases. So, assume $r=2$.
If $\ell=n-1$, there exists  $k\in \K_\ell$ such that $k e_{2}=e_4$.
From $xe_4= e_1$ we conclude that $\langle \K_\ell, \K_\ell ^x\rangle $  satisfies $(\mathrm{a}_n)$
and then it is a subgroup $\K_n$.
If $\ell=n-2$,  $\K_{n-2}$ contains a subgroup $T$ that induces $\SL_{n-3}(q)$ on the 
subspace $\langle e_3, S_{n-4}\rangle$.
Since  $x$ fixes $S_{n-4}$ and swaps $e_3$ with $e_2$,  it follows that $L=\langle T, T^x\rangle$ induces
$\SL_{n-2}(q)$ on $\langle e_2,e_3, S_{n-4}\rangle$.
Now, $y^{-1}$ fixes $e_2$ and $y^{-1}\langle  e_3, S_{n-4}\rangle=S_{n-3}$: this means that
$M=\langle L,L^y\rangle$ induces  $\SL_{n-1}(q)$ on $S_{n-1}$.
Finally, $\langle M,M^x\rangle =\K_n$.
\end{proof}

Given $g\in \Mat_\ell(\F)$, $\lambda \in \F$ and  a $g$-invariant subspace $L$ of $\F^{\ell}$, we define
\begin{equation}\label{eigen}
L_\lambda(g) =\{ w \in W \mid  gw = \lambda w\}. 
\end{equation}

\begin{lemma}\label{HT} 
Let $U$ be a proper $H$-invariant subspace of the symplectic space $\V$
and $\overline U$ be any $H^T$-invariant complement of  $U$. Let $\lambda$ be an eigenvalue
of some $g\in H$.
\begin{itemize}
\item[(1)] If $U$ has dimension $k$, then $\V$  has also  an $H$-invariant subspace of dimension $2n-k$. 
In particular we may assume $\dim U\leq n$.
\item[(2)] If the eigenspace $\V_{\lambda^{-1}}(g)$ is not contained
in any proper $H$-invariant subspace of $\V$, then $U\cap \V_\lambda(g)\neq \{0\}$.
\end{itemize}
\end{lemma}

\begin{proof}
(1)  From $H^T=H^J$ and $J^2=-\I_{2n}$ we get that the subspace $\overline{U}$ is $H^T$-invariant if and only if
 $J\overline{U}$ is $H$-invariant. So,  $J\overline{U}$ is an $H$-invariant subspace of dimension $2n-\dim(U)$.\\
(2) By sake of contradiction,  assume $U\cap \V_\lambda(g)=\{0\}$.
Then $g_{|U}$ does not have the eigenvalue $\lambda$, whence $\V_\lambda(g^T)\leq \overline{U}$.
From $g^T=\left(g^{-1}\right)^J$ we get $\overline U \geq \V_\lambda(g^T)=J\V_\lambda(g^{-1})=
J\V_{\lambda^{-1}}(g)$. Hence, $\V_{\lambda^{-1}}(g)$ is contained in
the proper $H$-invariant subspace $J\overline{U}$, a contradiction.
\end{proof}

\begin{lemma}\label{campoN}
Let $q=p^f$ with $f>1$,  and let $g(t)\in \F_p[t]$ be a monic polynomial of degree $s\geq 2$
such that $g(0)=0$.
Denote by $\N=\N(q)$ the number of elements $b \in \F_q^*$ such that
$\F_p[g(b)]\neq \F_q$. Then
$$\N\leq s \cdot p  \frac{ p^{\left \lfloor f/2\right \rfloor} -1}{p-1}.$$
\end{lemma}

\begin{proof}
For every $c\in \F_q$ there are at most $s$ elements $b\in \F_q^*$ such that $g(b)=c$.
Considering the possible orders of the subfields of $\F_q$ we have
$$\N\leq s\left(p+p^2+\ldots+ p^{\left \lfloor f/2\right \rfloor}\right)=s\cdot p\frac{ p^{\left \lfloor f/2\right \rfloor} 
-1}{p-1}.$$
\end{proof}

\begin{lemma}\label{minimal field} 
Let $\beta\in \F_{q^2}$ be primitive, namely $\F_p[\beta]=\F_{q^2}$, and let  $\alpha=\beta+\beta^q$.
If $\beta^{q+1}\in \F_p[\alpha]$, then $\alpha$ is a primitive element of $\F_q$.
Moreover, if $\gamma\in \F_{q^2}$ has order $q+1$, then $\gamma$ is primitive, 
and $\gamma^3$ is primitive in $\F_{q^2}$ whenever $q\neq 2,5,8$.
\end{lemma}

\begin{proof}
Clearly, $\beta$ is a root of $t^2-\alpha t+\beta^{q+1}$ with coefficients in $\F_p[\alpha]$. It follows 
$\left[\F_p[\beta]:\F_p[\alpha]\right] \leq 2$. From $\F_p[\beta]=\F_{q^2}$ and $\F_p[\alpha]\leq \F_q$ we get $\F_p[\alpha]=\F_q$.
Set $s_1=[\F_{q^2}: \F_p[\gamma]]$.  It follows that $(q+1)^{s_1}\leq q^2$, whence $s_1=1$, namely $\gamma$ is primitive.
Furthermore, the same happens if $(3,q+1)=1$. So, assume $q+1=3k$.
Since $\gamma$ satisfies $t^3-\gamma^3$ over $\F_p[\gamma^3]$, we have $s:=[\F_{q^2}: \F_p[\gamma^3]]\leq 3$.
If $s=3$, then $\frac{(q+1)^3}{27}\leq q^2$ whence $q\leq 23$.
Now, if $q$ is a prime, $\F_{q^2}$ does not have subfields of index $3$, giving the only possibility $q=8$.
If $s=2$, then $\F_p[\gamma^3]=\F_q$. It follows that $\frac{q+1}{3}=k$ divides $(q+1,q-1)$, whence $q=2,5$.
\end{proof}

Our direct computations are based on a result due to Liebeck, Praeger and Saxl.
Given a finite group $B$, let $\pi(B)$ be the set of the prime divisors of $|B|$.
For simplicity, if $b \in B$,  we write $\pi(b)$ for $\pi(\langle b \rangle)$.

\begin{lemma}\label{LPS}
Suppose that $A$ is a subgroup of a group $B$ such that  $\pi(A)=\pi(B)$.
\begin{itemize}
 \item[(1)] If $B=\PSL_n(q)$ with $n\geq 5$ and $(n,q)\neq (6,2)$,  then $A=B$.
 \item[(2)] If $B=\PSp_{2n}(q)$ with $n\geq 4$,  then either $A=B$ or $n$ and $q$ are even and $\Omega_{2n}^-(q)\unlhd A$.
\end{itemize}
\end{lemma}

\begin{proof}
It follows from \cite[Corollary 5 and Table 10.7]{LPS}.
\end{proof}

For the classification of the maximal subgroups of the finite quasi-simple classical groups, we refer to \cite{Ho,KL}.

\section{Case $n=4$}\label{caso n=4}

To acquaint the reader with our definitions of Section \ref{sec2}, we start with the case $n=4$, 
in which our generators, with respect to $\{e_1,\dots,e_4, e_{-1}\dots e_{-4}\}$, act
as showed in Figure \ref{gen4}.
Actually  Vasilyev and  Vsemirnov \cite{Vse2,Sp10} provided $(2,3)$-generators for
$\Sp_8(p)$ and $\Sp_{10}(p)$. However, aiming to uniform generators,
we use their result only for $p=2$. 

\begin{figure}[ht]
\begin{footnotesize}
$$x=
\begin{pmatrix}
 0 & 1 & 0 & 0 & 0 & 0 & 0 & 0\\
 1 & 0 & 0 & 0 & 0 & 0 & 0 & 0\\
 0 & 0 & -1 & a & 0 & 0 & 0 & 0\\
 0 & 0 & 0 & 1 & 0 & 0 & 0 & 0\\
 0 & 0 & 0 & 0 & 0 & 1 & 0 & 0\\
 0 & 0 & 0 & 0 & 1 & 0 & 0 & 0\\
 0 & 0 & 0 & 0 & 0 & 0 & -1 & 0\\
 0 & 0 & 0 & 0 & 0 & 0 & a & 1
\end{pmatrix},\quad
y=\begin{pmatrix}
 0 & 0 & 0 & 0 & -1 & 0 & 0 & 0\\
 0 & 0 & 0 & 1 & 0 & 0 & 0 & 0\\
 0 & 1 & 0 & 0 & 0 & 0 & 0 & 0\\
 0 & 0 & 1 & 0 & 0 & 0 & 0 & 0\\
 1 & 0 & 0 & 0 & -1 & 0 & 0 & 0\\
 0 & 0 & 0 & 0 & 0 & 0 & 0 & 1\\
 0 & 0 & 0 & 0 & 0 & 1 & 0 & 0\\
 0 & 0 & 0 & 0 & 0 & 0 & 1 & 0
\end{pmatrix}.$$
\end{footnotesize}
\caption{Generators for $n=4$ and $q>2$.}\label{gen4}
\end{figure}

A direct calculation gives that the characteristic polynomial $\chi_{[x,y]}(t)$ of $[x,y]$ is:
$$t^8 + 2t^7 + t^6 + 2t^5 + 4t^4 + 2t^3 + t^2 + 2t + 1=(t+1)^4   (t^2 - t + 1)^2.$$
Moreover, when $a\neq 0$, the eigenspace $\V_{-1}\left([x,y]\right)$ 
is generated by
$w_1$ and $xw_1=w_2$, where
$$w_1= (1,0,-a^{-1},0,0,0,-a^{-1},0)^T\equad w_2=(0,1, a^{-1} ,0,0,0,a^{-1},-1)^T.$$

\begin{lemma}\label{M=H} 
Suppose that $a \in\F_q^*$ satisfies all the following conditions:
$$\mathrm{(i)}\quad a^2+3\neq 0;\qquad \mathrm{(ii)}\quad a^4-a^2+4\neq 0.$$
Then the group $H=\langle x,y\rangle$ is absolutely irreducible.
\end{lemma}

\begin{proof}
Our first claim is that $\V_{-1}([x,y])$ is not contained in any proper $H$-invariant subspace of 
$\V=\F^8$. Indeed, set  $u_\pm=w_1\pm w_2$ and consider the matrix 
$M_\pm \in \Mat_8(\F)$ whose columns are the images of $u_\pm$ under
$\I_{8}, y,y^2, (y^2x)^2, (y^2x)^3, (y^2x)^4, (y^2x)^5,$ $(y^2x)^6$.
Then, $\det(M_+)=(a^2+3 )^4$ and $\det(M_-)=\frac{(a^4-a^2+4 )^4}{a^8}$.
Hence, under conditions (i) and (ii), our claim is proved.
Now, let $U$ be a proper $H$-invariant subspace of $\V$.
By Lemma \ref{HT}(2), we have $\V_{-1}([x,y])\cap U\neq \{0\}$.
As $xw_1=w_2$, this implies that  $U$ contains at least one of the vectors $u_{\pm}$.
By the above, we conclude that $U=\V$, and so  $H$ is absolutely irreducible.
\end{proof}

Conditions (i) and (ii) in the previous lemma are also necessary. 
In fact, for $u=u_+$ when $a^2+3=0$ and $u=u_-$ when $a^4-a^2+4=0$, the subspace
generated by $u$ and its images under $y,y^2,xy^2,yxy^2 $  is $H$-invariant.

We will need the characteristic polynomial of $xy^2$:
\begin{equation}\label{trace}
\chi_{xy^2}(t) = t^8 - a t^7 + a t^5 - (a^2 + 1) t^4 + a t^3 - a t + 1.
\end{equation}

\begin{lemma}\label{tensor primitive} 
If $H$ is absolutely irreducible, then it is primitive.
\end{lemma}

\begin{proof}
Suppose first that $H$ preserves a decomposition  $\Vb=\F_q^8=V_1\oplus \ldots \oplus V_\ell$, $\ell >1$. 
By the irreducibility, the permutation action must
be transitive. By \cite[Table 8.48]{Ho} we can reduce our analysis to the cases  $\ell\in\{2,4\}$.
Furthermore, we may assume $yV_1=V_1$ and $xV_1=V_2$.\\
\noindent \textbf{Case $\ell=2$.}
Then $yV_2=V_2$. It follows $\tr (xy^2)=0$ in contrast with  \eqref{trace}.\\
\noindent \textbf{Case $\ell=4$.} We may assume $V_3=yV_2$, $V_4=yV_3$ and
either (1) $xV_3=V_4$ or (2) $xV_3=V_3$ and $xV_4=V_4$.
In case (1), $xy^{2}$ induces the permutation $(V_1,V_2,V_3)$. Hence, with respect to a basis conform to
the above decomposition,  it acts on $V$ as
$$\begin{pmatrix}
0&0&\I_2&0\\
\I_2&0&0&0\\
0&\I_2&0&0\\
0 & 0 & 0& B
\end{pmatrix},\quad  B=\begin{pmatrix}
b_1&b_2\\
b_3&b_4
\end{pmatrix}.$$
Equating the coefficient of the term $t^2$ of the characteristic polynomials of this matrix and of $xy^{2}$,
we get the contradiction $0=1$.
In case (2), $xy^{2}$ induces the permutation $(V_1,V_2, V_4,V_3)$. Reasoning as above we get
 the contradiction $0=\tr (xy^{2})=a$. 
\end{proof}

The characteristic polynomial of $[x,y]^3$ is $(t+1)^8$ and the eigenspace
$$\Vb_{-1}\left([x,y]^3\right)=\langle (\alpha_1, \alpha_2,\alpha_3,\alpha_4,\alpha_5,-\alpha_4-\alpha_5,
a\alpha_5+\alpha_3,\alpha_6)^T: \alpha_i \in \F_q\rangle $$
has dimension $6$.
It follows that $[x,y]^3$ is a bireflection and $[x,y]^{3p}=-\I_8$.

\begin{lemma}\label{SL3odd}
Suppose $p$ odd. Let $a\in \F_q^*$ be such that
$$\mathrm{(a)}\quad \F_p[a^4-a^2]=\F_q ;\qquad \mathrm{(b)}\quad 3a^2+1\neq 0.$$
Then $H$ contains a subgroup having $\SL_3(q)$ as epimorphic image.
\end{lemma}

\begin{proof}
By the above computations, $\tau=[x,y]^6$ is a bireflection.
Taking $v=e_1+e_2-e_{-4}$, the subgroup $G=\left\langle \tau, \tau^y , \tau^{y^2}\right\rangle$ fixes the subspace
$U=\left\langle v, y^2 v, y v\right\rangle$, acting as

\begin{footnotesize}
$$\left\langle\begin{pmatrix}
1& 2(a^2-a) & -2(a^2+a)\\
0&1&0\\
0&0&1
\end{pmatrix}, \ \begin{pmatrix}
1&0&0\\
-2(a^2+a) & 1 & 2(a^2-a) \\
0&0&1
\end{pmatrix}, \ \begin{pmatrix}
1 &  0 & 0\\
0 & 1 & 0\\
2(a^2-a) & -2(a^2+a) & 1
\end{pmatrix}\right\rangle.$$
\end{footnotesize}

Suppose $q\neq 9$. If $a\neq \pm 1$, by Dickson's Lemma  (see \cite[Theorem 8.4]{G}) and assumptions (a), the group 
$\left\langle \tau, \tau^{y}\right\rangle$ induces $\SL_2(q)$ on $\langle v, y^2 v\rangle$.
Similarly, if $a=\pm 1$ (hence $q=p$),  the group  $G_{\pm 1}= \left\langle [\tau,\tau^{y^{\pm 1}}], \tau^{y^{\mp 1}}\right\rangle $
induces $\SL_2(q)$ on the subspace $\langle  y^2v, yv \rangle$.
By assumption (b), the group $G_{|U}$ is an  irreducible subgroup of
$\SL_3(q)$. Lemma \ref{McLaughlin} gives  $G_{|U}=\SL_3(q)$. 
If $q=9$, direct calculations lead to the same conclusion.
\end{proof}

\begin{theorem}\label{main4}
If $p=2<q$, suppose that $a\in \F_q^*$ is such that
$\F_q=\F_2[a]$;
if $p$ is odd, suppose that $a\in \F_q^*$ satisfies 
$$\F_q=\F_p[a^4-a^2]\equad (a^2+3)(a^4-a^2+4)(3a^2+1)\neq 0.$$
Then $H=\Sp_8(q)$. In particular $\Sp_8(q)$ is $(2,3)$-generated for all $q> 2$.
\end{theorem}  

\begin{proof} 
First of all, there exists $a\in \F_q^*$ satisfying our hypotheses.
For $p=2$ it suffices to take as $a$ any primitive element of $\F_q^*$. 
Next, suppose $p$ odd. There are at most $2+4+2=8$ elements of $\F_q^*$ to be excluded
in order to satisfy the required conditions.
If $q=p$, it is easy to check that $a=1$ satisfies all requirements.
If $q=p^f$ with $f>1$, call $\N(q)$ the number of elements $b\in \F_q^*$ 
such that $\F_p[b^4-b^2]\neq \F_q$. Then, our claim is true whenever $p^f-1-\N(q) > 0$,
as we can drop the other conditions on $a$.
By Lemma \ref{campoN}, we have $\N(q)\leq 4 p\frac{p^{\left\lfloor f/2\right \rfloor}-1}{p-1}$ and 
so it suffices to check when $p^f- 4p\frac{p^{\left\lfloor f/2\right \rfloor}-1}{p-1}>1$.
This holds unless $p=3$ and $f=2$.
If $q=9$, take  $a\in\F_9$ whose minimal polynomial over $\F_3$ is $t^2-t-1$.

Now, by Lemmas \ref{M=H} and \ref{tensor primitive} the group $H$ is absolutely irreducible and  primitive.
Suppose that $H$ is contained in a maximal subgroup $M$ of $\Sp_8(q)$, whose classification is
described in \cite[Tables 8.48 and 8.49]{Ho}.
The classes $\mathcal{C}_4,\mathcal{C}_6,\mathcal{C}_7$ appear in the list of the
maximal subgroups of $\Sp_8(q)$ only when $q$ is odd.
However, in this case, the order of $H$ is divisible by the order of $\SL_3(q)$ and so we can easily exclude
these possibilities by order reasons.
In particular, $H$ is  tensor indecomposable. 
Since $H$  contains the bireflection $[x,y]^3$ we can apply  \cite[Theorem 7.1]{GS}, which allows to reduce our analysis
to the classes $\mathcal{C}_5$ and $\mathcal{C}_8$.

So, suppose that $M\in \mathcal{C}_5$. Then, for every $h \in H$  we have $h^2\in \Sp_8(\F_{q_0})$ for some subfield
$\F_{q_0}$ of $\F_{q}$.
In particular, we get
$a^2=\tr ((xy)^2) \in \F_{q_0}$. We conclude $q_0=q$ 
by the assumption $\F_q\leq \F_p[a^2]$.
Finally, suppose that  $M\in \mathcal{C}_8$. Then $q$ is even and  $H$ should preserve a quadratic form $Q$, whose corresponding bilinear
form $f_Q(v,u)=Q(v+u)+Q(v)+Q(u)$ has Gram matrix $J$ with respect to the canonical basis. It follows 
$$Q\left(\sum_{i\in I} z_i e_i \right)=\sum_{i\in I} Q(e_i)z_i^2+ \sum_{i=1}^4 z_i z_{-i},\quad 
\textrm{ where } I=\{\pm 1, \pm 2 , \pm 3, \pm 4\}.$$
Imposing that $Q$ is preserved by $y$ we get $Q(e_{-1})=Q(ye_{-1})=Q(e_1)+Q(e_{-1})+1$, whence
$Q(e_1)=1$. Also,  $Q(e_2)=Q(ye_2)=Q(e_3)$.
Next, imposing that $Q$ is preserved by $x$, we get $Q(e_1)=Q(xe_1)=Q(e_2)$ and 
$Q(e_4)=Q(xe_4)=a^2Q(e_3)+Q(e_4)$, whence  $Q(e_3)=0$.
This gives the  contradiction $1=Q(e_1)=Q(e_3)=0$. We conclude that $H=\Sp_8(q)$.
\end{proof}

\section{Case $n=5$}\label{case5}

By the result of  Vasilyev and  Vsemirnov \cite{Sp10} we may assume $q>2$.
Define $x,y$  as in Figure \ref{gen5}: their respective orders are $2$ and $3$.
Since $x^T J x  = y ^T J y = J$, it follows that $H=\langle x,y\rangle$ is contained in $\Sp_{10}(q)$.

\begin{figure}[ht]
\begin{footnotesize}
$$x=
\begin{pmatrix}
0 & 0 & 1 & 0 & 0 &  0 & 0 & 0 & 0 & 0\\
0 & 1 & 0 & 0 & 0 &  0 & 0 & 0 & 0 & 0\\
1 & 0 & 0 & 0 & 0 &  0 & 0 & 0 & 0 & 0\\
0 & 0 & 0 & -1 & a &  0 & 0 & 0 & 0 & 0\\
0 & 0 & 0 & 0 & 1 &  0 & 0 & 0 & 0 & 0\\
0 & 0 & 0 & 0 & 0 &  0 & 0 & 1 & 0 & 0\\
0 & 0 & 0 & 0 & 0 &  0 & 1 & 0 & 0 & 0\\
0 & 0 & 0 & 0 & 0 &  1 & 0 & 0 & 0 & 0\\
0 & 0 & 0 & 0 & 0 &  0 & 0 & 0 & -1 & 0\\
0 & 0 & 0 & 0 & 0 &  0 & 0 & 0 & a & 1
\end{pmatrix},\quad
y=
\begin{pmatrix}
 0 & 0 & 0 & 0 &  0 & -1 & 0 & 0 & 0 & 0\\
 0 & 0 & 0 & 1 &  0 &  0 & 0 & 0 & 0 & 0\\
 0 & 1 & 0 & 0 &  0 &  0 & 0 & 0 & 0 & 0\\
 0 & 0 & 1 & 0 &  0 &  0 & 0 & 0 & 0 & 0\\
 0 & 0 & 0 & 0 &  0 &  0 & 0 & 0 & 0 & -1\\
 1 & 0 & 0 & 0 &  0 & -1 & 0 & 0 & 0 & 0\\
 0 & 0 & 0 & 0 &  0 &  0 & 0 & 0 & 1 & 0\\
 0 & 0 & 0 & 0 &  0 &  0 & 1 & 0 & 0 & 0\\
 0 & 0 & 0 & 0 &  0 &  0 & 0 & 1 & 0 & 0\\
 0 & 0 & 0 & 0 &  1 &  0 & 0 & 0 & 0 & -1
 \end{pmatrix}.$$
\end{footnotesize}
\caption{Generators for $n=5$ and $q>2$.}\label{gen5}
\end{figure}

\begin{lemma}\label{irr5}
Let $a\in \F_q^\ast$ be such that $a^2+3\neq 0$. Then $H$ is absolutely irreducible.
\end{lemma}

\begin{proof} 
Let $U$ be a proper $H$-invariant subspace of $\V=\F^{10}$.
The eigenspace $\V_{-1}([x,y])$ is  generated by 
$w_1$ and  $w_2$, where
$$w_1=e_1+e_3-e_{-2} \equad w_2=e_{4}.$$
Suppose $w_2 \in U$. Since the matrix whose columns are the images of $w_2$ under
$$\I_{10},\; y,\; y^2,\; xy^2, \;(yx)^4,\; x(yx)^4,\; y(yx)^4, \;(yx)^5, \; [y^2,x],\;x[y^2,x]$$
has determinant $a^2$, we conclude that $U=\V$.
It follows that $\V_{-1}([x,y])$ is not contained in any proper $H$-invariant subspace of $\V$.
By Lemma \ref{HT}(2), we have $\V_{-1}([x,y])\cap U=\langle u \rangle$,
where  $u=w_1+\beta w_2$ for a suitable $\beta \in \F$.
By the same lemma, we may assume $\dim(U)\leq 5$.
Since $a^2+3\neq 0$, the matrix $M$ whose columns are the images of $u$ under $\I_{10}, y,  y^2, xy^2,$ $yxy^2,(xy^2)^2$ 
has rank $6$.
In fact,  the submatrix of $M$ corresponding to the rows $1, 2, 5, 7, 8, 10$
has determinant  $(a\beta)^2\neq 0$, provided $\beta\neq 0$.  
If $\beta=0$, the submatrix of $M$ corresponding to the rows $3, 4, 6, 8, 9, 10$ has determinant $a(a^2 + 3)\neq 0$.  
So $M$ has rank $6$, a  contradiction. 
\end{proof}

The element $\tau=[x,y]^{6}$ has characteristic polynomial $(t-1)^{10}$ and  fixed point space $\langle e_1, e_3, e_4, e_5, e_{-2}, e_{-5}, e_2- e_{-1}, e_2-e_{-3} \rangle$,
since $a\neq 0$. So $\tau$ is a bireflection of order $p$.

\begin{lemma}\label{SL3-5}
Assume $q\neq 2,4$. Let $U=\langle e_2,e_3,e_4\rangle$ and  $G=\langle \tau,\tau^y, \tau^{y^2}\rangle$.
If  $\F_q=\F_p[a^6]$, then $G_{|U}=\SL_3(q)$.
\end{lemma}

\begin{proof}
The group $G$ fixes  $U=\langle -a^2 e_2, e_3,e_4\rangle$ acting 
as 
$$\left\langle \begin{pmatrix}
1 & 0 & 0\\
0 & 1 & 0 \\
a^4 & 0 & 1
          \end{pmatrix},
     \begin{pmatrix}
1 & 0 & 0\\
0 & 1 & -a^2 \\
0 & 0 & 1
 \end{pmatrix},
\begin{pmatrix}
1 & 1 & 0\\
0 & 1 & 0 \\
0 & 0 & 1
          \end{pmatrix}               
\right\rangle.$$
Then, $\eta=[\tau,\tau^{y}]$ acts on $U$
as $\I_3+a^6E_{2,1}$.
Also, $G_{|U}$  is an irreducible subgroup of $\SL_3(q)$.
Suppose $q\neq 9$ is odd. By Dickson's Lemma and the assumption $\F_q=\F_p[a^6]$, the group
$\langle \eta, \tau^{y^2}\rangle $ induces $\SL_2(q)$ on $\langle e_2,e_3\rangle$: 
by Lemma \ref{McLaughlin} 
we obtain $G_{|U}=\SL_3(q)$. The same holds also for $q=9$ by direct computations.
Suppose now that $q>4$ is even.
The group $G_{|U}$ contains two transvections, $\tau$ and $\eta$, having the same 
axis but distinct centers: by \cite{Pip} we obtain $G_{|U}=\SL_3(q_0)$ for some $q_0\leq q$.
Since $\tr((\eta\tau^{y^2})_{|U})=a^6+1$, we get $G_{|U}=\SL_3(q)$.
\end{proof}

\begin{lemma}\label{ex5}
Suppose $q\not \in \{2,4,25\}$. Then there exists $a\in \F_q^*$  such that 
$$a^2+3\neq 0 \equad \F_p[a^6]=\F_q.$$
\end{lemma}

\begin{proof}
If $q=p\neq 2$, then $a=1$ satisfies all the requirements.
If $q=p^f$ with $f>1$, call $\N(q)$ the number of elements $b\in \F_q^*$ 
such that $\F_p[b^6]\neq \F_q$. Then, our claim is true whenever $p^f-1-\N(q) > 0$,
as we can drop the other condition.
By Lemma \ref{campoN}, we have $\N(q)\leq 6p\frac{p^{\left\lfloor f/2\right \rfloor}-1}{p-1}$ and 
so it suffices to check when $p^f- 6p\frac{p^{\left\lfloor f/2\right \rfloor}-1}{p-1} > 1$.
This holds unless $q \in \{4,8,16,32,64, 9,25 \}$. For $q\neq 4,25$ one can take $a\in \F_q^*$ whose minimal polynomial $m(t)$
over $\F_p$ is as follows:
$$\begin{array}{|c|c|c|c|c|c|}\hline
q & 8 & 16 & 32 & 64 & 9 \\\hline
m(t) & t^3+t+1 &  t^4+t^3+1&  t^5+t^2+1 &  t^6+t+1 & t^2+t-1 \\
\hline   \end{array}$$
\end{proof}

\begin{theorem}\label{main5}
If $a\in \F_q^*$ is such that $a^2+3\neq 0$ and $\F_p[a^6]=\F_q$,
then $H=\Sp_{10}(q)$. In particular $\Sp_{10}(q)$ is $(2,3)$-generated for all $q> 2$.
\end{theorem}  

\begin{proof} 
Suppose there exists $a\in \F_q^*$ which satisfies all the requirements (by Lemma \ref{ex5} this is true if $q\not \in \{2,4,25\}$).
For any such value of $a$, the subgroup $H$ is absolutely irreducible and its order is divisible by
the order of $\SL_3(q)$, see Lemmas \ref{irr5} and \ref{SL3-5}. If our claim is false, there exists  a maximal subgroup $M$ of 
$\Sp_{10}(q)$ which contains $H$.
Considering the classification of the maximal subgroups of $\Sp_{10}(q)$ (see \cite[Tables 8.64 and 8.65]{Ho}) 
we obtain that $M$ belongs to one of the classes $\Cl_2, \Cl_4, \Cl_5, \Cl_8, \ClS$. 
By order reasons, we can exclude  
$M\cong \Sp_2(q)^5:\Sym(5) \in \Cl_2$ and $M\cong \Sp_2(q)\circ \Or_5(q) \in \Cl_4$.

Suppose that $M\in \mathcal{C}_5$. Then, for every $h \in H$  we have $h^2\in \Sp_{10}(\F_{q_0})$ for some subfield
$\F_{q_0}$ of $\F_{q}$. In particular, if $p=2$ we get
$a^4 + 1 =\tr ((xy)^{10}) \in \F_{q_0}$; if $p$ is odd, we get
$-8a^2 - 5=\tr ((xy)^{8}) \in \F_{q_0}$. 
In both cases, we conclude $q_0=q$, by the assumption $\F_q=~\F_p[a^6]$.

Next, suppose that $M$ is imprimitive. By the above considerations on the class $\Cl_2$, 
$H$ preserves a decomposition $\F^{10}=V_1\oplus V_2$ with $\dim V_1=\dim V_2=5$.
Clearly, we must have  $yV_1=V_1$, $yV_2=V_2$ and $xV_1=V_2$. Hence, the product $xy$ should have trace equal
to zero, but $\tr(xy)=-1$.
We conclude that $M\not \in \Cl_2$. In particular, $H$ is primitive and tensor indecomposable.
Since it contains the bireflection $\tau$, by \cite[Theorem 7.1]{GS} it can only be a group of the items
(a) and (d) of the list given there. 
This means that either $H\leq \SO_{10}^\pm(q)$ and $q$ is even, or $H$ normalizes $\SU_5(4)$ and $q=p$ is odd.
In the latest case, $\pi(M)=\{2, 3, 5, 11 \}$.
Since $\tau$ has order $p$,  we easily obtain a contradiction unless  $p \in \{3,5,11\}$.
However, for these values of $p$ there exists a prime $s\geq 13$ which divides
$|\SL_3(p)|$, a contradiction.

Finally, suppose that $q$ is even and $H$ preserves a quadratic form $Q$, whose corresponding bilinear
form  has Gram matrix $J$ with respect to the canonical basis. 
Imposing that $Q$ is preserved by $y$ we get $Q(e_1)=1$ and $Q(e_3)=Q(e_4)$.
Next, imposing that $Q$ is preserved by $x$, we get $Q(e_1)=Q(e_3)$ and $Q(e_4)=0$,
whence the contradiction $1=Q(e_1)=Q(e_4)=0$. We conclude that $H=\Sp_{10}(q)$.

We are left to consider the exceptional cases $q\in \{4,25\}$.
For these values of $q$, take $a\in \F_q$ whose minimal polynomial over $\F_p$ is $t^2+t+1$.
Define $L_4=\{ 4, 7, 10, 12, 21 \}$ and $L_{25}=\{  1,6,  11,    12,13,  14 \}$.
Then 
$$\bigcup_{k \in L_q} \pi( (xy)^k y )=\pi(\Sp_{10}(q)).$$
By Lemma \ref{LPS} we conclude that  $H=\Sp_{10}(q)$.
\end{proof}

\section{Action of the commutator, $n= 10$ or  $n\geq 12$}\label{azione}
 
In this section we assume  $n= 10$ or  $n\geq 12$.
Setting $T(\pm i)=\langle e_{\pm i}, e_{\pm (i+4)},$
$e_{\pm (i+5)}\rangle$,
we can decompose  $\Vb=\F_q^{2n}$ into the  direct sum of the following
$[x,y]$-invariant subspaces: 
$$\begin{array}{rcl}
\mathcal{A}_0& =& \left\{\begin{array}{ll}
\langle e_{2}, e_{3}, e_{4},e_{6}, e_{7}, e_{-1}\rangle \oplus 
\langle e_{1}, e_{-2}, e_{-3}, e_{-4},e_{-6}, e_{-7}\rangle ,& \textrm{if } r=0,\ p>2,\\
\langle e_1,e_{-1}\rangle \oplus \langle e_3,e_4,e_{-3},e_{-4}\rangle
\oplus \langle e_2,e_6,e_7,e_{-2},e_{-6},e_{-7}\rangle , & \textrm{if } r=0,\  p=2,
\end{array}\right.\\[-8pt]\\
\mathcal{A}_1 & =& \langle   e_{1}, e_{2}, e_{4},e_{5},  e_{-1}, e_{-2}, e_{-4},e_{-5}\rangle, \ \textrm{if } r=1,\\[-8pt]\\
\mathcal{A}_2 & =&  \langle e_1,e_2,e_3,e_4,e_5,e_6,e_8,e_9, e_{-1},e_{-2},e_{-3},e_{-4},e_{-5},e_{-6},e_{-8},e_{-9}\rangle , 
\ \textrm{if } r=2,\\[-8pt]\\
\mathcal{B}& =& \left\{
\begin{array}{ll}
\bigoplus\limits_{j=0}^{m-5} \left(T(5+3j) \oplus T(-(5+3j))\right) & \textrm{if } r=0,\\[5pt]
\bigoplus\limits_{j=0}^{m-4} \left(T(3+3j) \oplus T(-(3+3j))\right) & \textrm{if } r=1,\\[5pt]
\bigoplus\limits_{j=0}^{m-5} \left(T(7+3j) \oplus T(-(7+3j))\right) & \textrm{if } r=2,
\end{array}
\right.\\[-8pt]\\
\mathcal{C}^\pm  &=& \langle e_{\pm(n-7)},e_{\pm(n-4)},e_{\pm(n-3)},e_{\pm(n-2)},e_{\pm (n-1)},e_{\pm n}\rangle.
\end{array}$$

\begin{lemma} 
If $n=14$ and $p$ is odd, the order of $[x,y]_{|{\mathcal A}_2}$ is $16$; in the other cases, 
we have $\left([x,y]_{|{\mathcal A}_r}\right)^{24}=\I$.
Moreover, $\left([x,y]_{|\mathcal B}\right)^6=\id$.   
\end{lemma}

\begin{proof}
For the action of $[x,y]$ on $\mathcal{A}_r$ we have the following cases:
\begin{itemize}
\item  $r=0$, $p$ odd. 
Let $\varepsilon=-1$ if $n=12$ and  $\varepsilon=+1$ if $n>12$.
The action on the two summands of $\mathcal{A}_0$ is represented by the respective following matrices of order $9-3\varepsilon$:

\begin{footnotesize}
$$\begin{pmatrix}
 0 & -1 & 0 &  0 & 0 & 0 \\ 
 0 & 0 & 1 &  0 & 0 & 0 \\
 0 & 0 & 0 &  0 & 0 & 1 \\
 0 & 0 & 0 &  0 & \varepsilon & 0 \\
 1 & 0 & 0 &  0 & 0 & 0 \\ 
 0 & 0 & 0 &  -1 & 0 & 0
\end{pmatrix},
\quad \begin{pmatrix}
 0 & 0 & 0 &  0 & 1 & 0\\
 0 & 0 & -1 &  0 & 0 & 0\\
 0 & 0 & 0 &  1 & 0 & 0\\
 -1 & 0 & 0 &  0 & 0 & 0\\
 0 & 0 & 0 &  0 & 0 & \varepsilon\\
 0 & 1 & 0 &  0 & 0 & 0
\end{pmatrix}.$$
\end{footnotesize}

\item  $r=0$, $p=2$.
The action on the  three summands of $\mathcal{A}_0$ is represented by 
the following matrices of order, respectively, $3$, $4$ and $6$:

\begin{footnotesize}
$$\begin{pmatrix}
0  & 1\\
1  & 1
\end{pmatrix},\quad \begin{pmatrix}
    0  &     1 &      0 &      0\\
    1   &    0 &      0 &      0\\
    0   &    0 &      0 &      1\\
    1   &    0 &      1 &      0
\end{pmatrix},\quad \begin{pmatrix}
    0 &      0  &     0   &    0 &      1 &      0\\
    0 &      0  &     1  &     0 &      0 &      0\\
    1 &      0  &     0   &    0 &      0 &      0\\
    0  &     1  &     0   &    0 &      0 &      0\\
    0  &     0    &   0  &     0 &      0 &      1\\
    0   &    0   &    0 &      1 &      0 &      0
\end{pmatrix}.$$
\end{footnotesize}

\item  $r=1$. 
Let $\varepsilon=-1$ if $n=10$ and $p$ is odd;  $\varepsilon=+1$ otherwise.
The action on $\mathcal{A}_1$ is represented by the following square matrix
of size $8$ and order $6-2\varepsilon$: $E_{1,2}+E_{2,7}+E_{5,6}+E_{8,1}-(E_{2,3}+E_{4,5}+ E_{6,3}+E_{8,5})+
\varepsilon (E_{3,4}+E_{7,8})$.

\item $r=2$.
Let $\varepsilon=-1$ if $n=14$ and $p$ is odd;  $\varepsilon=+1$ otherwise.
The action on $\mathcal{A}_2$ is represented by the following square matrix of  size $16$ and
order $12-4\varepsilon$:
$E_{1,7}+E_{3,5}+E_{4,11}+E_{5,6}+E_{6,2}+E_{8,4}+E_{9,15}+E_{10,1}+E_{11,13}+E_{13,14}+E_{14,10}+E_{16,12}
-(E_{2,9}+E_{4,3}+E_{10,9}+E_{12,3})+\varepsilon (E_{7,8}+E_{15,16}) $. 
\end{itemize}

As to the action on $\mathcal B$ we note that $[x,y]$ fixes each of the spaces $T(\pm i)$, inducing a  monomial 3-cycle, with entries 
$\pm 1$.
\end{proof}

Setting $\vartheta =\vartheta_1$ if $p$ is odd, 
$\vartheta=\vartheta_2$ if $p=2 < q$, and 
$\vartheta=\vartheta_3$ if $q=2$ with
\begin{footnotesize}
$$\vartheta_1=
\begin{pmatrix}
  0 &    0 &    1 &    0 &    0 &    0\\
  0 &    0 &    0 &    0 &    0 &   -1\\
  0 &    0 &    0 &   -1 &    0 &    0\\
  1 &    0 &    0 &    0 &    0  &   0\\
  0 &    0 &    0 &   -a &   -1  &   0\\
  0  &   1 &    0 &  a^2 &    a  &   0 
\end{pmatrix},\quad
\vartheta_2=
\begin{pmatrix}
    0  &   0 &    1 &    0  &   0   &  0\\
    0  &   0 &    0 &    1  &   1   &  1\\
    0  &   0 &    0 &    1  &   0   &  0\\
    1 &    0 &    0 &    0  &   0   &  0\\
    0 &    a &    0 &  a^2  &  a + 1 &    0\\
    0 &    1 &    0 & a + 1 &    1   &  0
        \end{pmatrix},
        $$
$$\vartheta_3=
\begin{pmatrix}      
0 & 0 & 1 & 0 & 0 & 0\\
0 & 0 & 0 & 1 & 1 & 0\\
0 & 0 & 0 & 0 & 1 & 1\\
1 & 0 & 0 & 0 & 0 & 0\\
0 & 0 & 0 & 1 & 0 & 0\\
0 & 1 & 0 & 1 & 1 & 1
\end{pmatrix},$$
\end{footnotesize}

\noindent the element $[x,y]$ acts on $\mathcal{C}^+\oplus \mathcal{C}^-$ as
$\diag\left(\vartheta,\vartheta^{-T}\right)$. 
The characteristic polynomial of $\vartheta$ is
$$\chi_\vartheta(t)=\left\{\begin{array}{ll} (t + 1)^2(t^2 - t + 1) (t^2+1) & \textrm{ if } p \textrm{ is odd},\\
   (t + 1)^2(t^2 + t + 1)(t^2 + at + 1)  & \textrm{ if } p=2<q,\\
   (t + 1)^2(t^4 + t^3 + t^2 + t + 1)  & \textrm{ if } q=2.\\
                           \end{array}\right. 
$$
Define
\begin{equation}\label{deftau}
\tau=\left\{\begin{array}{ll}
{} [x,y]^{24(1-p)} & \textrm{ if } n=14 \textrm{ and } p>2,\\
{} [x,y]^{24} & \textrm{ otherwise.}
             \end{array}\right.
\end{equation}
From the previous analysis of $[x,y]$, we see that  $\tau$ is an element of $\wSL_n(q)$ that fixes $S_9$ and
such that $\tau e_i=e_i$ for all $1\leq i\leq n-9$. As $yS_9=S_9$, we have
\begin{equation}\label{S9}
G=\left\langle \tau , \tau^y, \tau^{y^2}\right\rangle \leq \wSL_9(q).
\end{equation}

\section{The general case}\label{sec4}

In this section we need $n$ to be sufficiently large.
So, if $q=2$ we will assume $n\geq 6$; 
if $q>2$ we will assume $n=10$ or $n\geq 12$.
In both cases we prove the $(2,3)$-generation of~$\Sp_{2n}(q)$.

\subsection{Subcase $q=2$ and $n \geq 6$}

Recall that, given a finite group $B$, we denote by $\pi(B)$ the set of the prime divisors of $|B|$.
Furthermore, if $b \in B$,  we write $\pi(b)$ for $\pi(\langle b \rangle)$.

\begin{proposition}\label{q=2}
For all $n \geq 6$ we have $H=\Sp_{2n}(2)$.
\end{proposition}

\begin{proof}
Let $L_6=\{ 2,5,11,14,22\}$ and $L_8=\{1,4,5,8,9,10,11,13\}$.
Direct computations show that, for $n\in \{6,8\}$,
$$\bigcup_{k \in L_n} \pi([x,y]^k y  ) =\pi(\Sp_{2n}(2)).$$
For $n\in \{7,9,11\}$, define
$L_7=\{1,3,9,11,13,18,22 \}$,
$L_9= \{3,7,8,9,10,12,15,39,114 \}$ 
and $L_{11}=\{ 1,3,5,10,13,14,15,37,51,99\}$.
Then 
$$\bigcup_{k \in L_n} \pi( (xy)^k y ) =\pi(\Sp_{2n}(2)).$$
By Lemma \ref{LPS} we conclude that either $H=\Sp_{2n}(2)$ or 
$n=6,8$ and $\Omega_{2n}^-(2) \unlhd H$.

So, assume $n\in \{6,8\}$ and suppose that $H$ preserves a non-degenerate quadratic form $Q$,
whose corresponding bilinear form $f_Q$ has Gram matrix $J$ with respect to the canonical basis.
For $n=6$, since $Q$ is preserved by $x$,  we get $Q(e_{-1})=1$ and $Q(e_{-3})=Q(e_{-4})$.
On the other hand, since $Q$ is preserved also by $y$, we have $Q(e_{-1})=Q(e_{-3})$, $Q(e_{-4})=Q(e_{-5})$
and $Q(e_{-4})=Q(e_{-4})+Q(e_{-5})$. So, we get the contradiction $1=Q(e_{-1})=Q(e_{-5})=0$.
For $n=8$, since $Q$ is preserved by $x$,  we get $Q(e_{1})=Q(e_4)$, $Q(e_5)=Q(e_6)$ and  
$Q(e_7)=0$. Considering the action of $y$, we obtain $Q(e_1)=1$, $Q(e_4)=Q(e_5)$, $Q(e_6)=Q(e_7)$,
whence the contradiction $1=Q(e_1)=Q(e_7)=0$.
Thus, in both case, $H$ is not contained in $\SO_{2n}^\pm(2)$, proving that $H=\Sp_{2n}(2)$.

Now, assume $n= 10$ or $n\geq 12$.
Consider the group $G\leq \wSL_9(2)$ defined in \eqref{S9}.
Take the element $g=\tau(\tau^{y^2})^{-1} \tau^y$, where $\tau=[x,y]^{24}$.
Since  
$$\bigcup_{k \in L} \pi(g^k \tau^y (\tau^{y^2})^{-1} ) =\pi(\SL_9(2)),\quad  \textrm{where } L= \{ 2,4,8,11,12\},$$
by Lemma \ref{LPS} we obtain $G=\wSL_9(2)$.
This means that $G$ is a subgroup $\K_9$: the statement follows from Theorem \ref{wSLell}.
\end{proof}

\subsection{Subcase $p=2 < q$, $n=10$ or $n\geq 12$} 

If $q\neq 8$,  we take $\sigma\in \F_{q^2}^*$ of order $q+1$ so that $\sigma +\sigma^{-1}=\sigma+\sigma^q\in \F_q$.
We define
\begin{equation}\label{a sigma}
a:=\sigma+\sigma^{-1}\quad \textrm{ where } \quad
\quad \left\{\begin{array}{ll}
\sigma\in \F_{8}^* \textrm{ has order } 7 & \textrm{ if } q= 8,\\
\sigma\in \F_{q^2}^* \textrm{ has order } q+1 & \textrm{ if } q\neq 8.
             \end{array}\right.
\end{equation}
We have $a^3+a= (\sigma +\sigma^{-1})^3 + \sigma +\sigma^{-1}=\sigma^3+\sigma^{-3}$.
Now, $\sigma^3$ is a primitive element of $\F_8$ when $q=8$, and a primitive element of $\F_{q^2}$ when $q\neq 8$, 
by Lemma \ref{minimal field}.
Applying  one more time the same lemma, we obtain that $\F_2[a^3+a]=\F_q$.

The characteristic polynomial of $\vartheta$ (see Section \ref{azione}) is  
$(t+1)^2(t^2+t+1)(t^2+at+1)$, with
roots $\sigma^{\pm 1}$ of multiplicity $1$, as  $\sigma^3\neq 1$.
Since $\vartheta^2$ has a fixed point space of dimension $2$, $\vartheta^{24}$ is conjugate in $\GL_6(q^2)$ to
$\diag(\rho,\rho^{-1},\I_4)$ with $\rho=\sigma^{24}\neq 1$.
We get that $\tau=[x,y]^{24}$ fixes $V$ and $\tau_{|V}$ is conjugate
in $\GL_n(q^2)$ to $\diag(\rho,\rho^{-1},\I_{n-2})$. In particular, 
$\tau_{|V}$ has order $7$ if $q=8$, $\frac{q+1}{\gcd(q+1,3)}$ otherwise.
The eigenspaces $V_{\rho^{\pm 1}}(\tau)$ and $V_{\rho^{\pm 1}}(\tau^T)$
(with the notation \eqref{eigen}) are contained in $S_8$, being generated, respectively, by:
\begin{equation}\label{s sigma}
\left\{\begin{array}{l}
s_{\rho^{\pm 1}}  =   e_{n-4}+ (1+\sigma^{\pm 1}) e_{n-1}+e_n,\\
\overline{s}_{\rho^{\pm 1}}  =  e_{n-7}+e_{n-4}+e_{n-1} + \sigma^{\mp 1}\left(e_{n-3}+e_{n}\right)
+ \sigma^{\pm 1}e_{n-2}. 
\end{array}\right.
\end{equation}
Using \eqref{S9} and the fact that $x S_{10}=x_2 S_{10}=S_{10}$, we have
\begin{equation}\label{K2}
K:=\left \langle  \tau^{y^i}, \tau^{y^j x} \mid i,j=0,\pm 1 \right \rangle=\langle G,G^x\rangle\leq \wSL_{10}(q)
\end{equation}
Furthermore, $K$ fixes the subspace $S_7$ and satisfies condition $(\mathrm{b}_7)$ of Definition \ref{Kell}.

\begin{lemma}\label{G7}
Take $a\in \F_q^*$ as in \eqref{a sigma} and  $K$ as in \eqref{K2}. Then, the restriction of $K$ to $S_7$ is the group $\SL_7(q)$.
\end{lemma}

\begin{proof}
The group $G_2=\langle \tau, \tau^y\rangle$ fixes the subspace $U_4=\langle s_\rho, s_{\rho^{-1}}, 
y^2s_\rho, y^2 s_{\rho^{-1}}\rangle$,
acting as $\Gamma=\langle \gamma_1, \gamma_2\rangle$, where
$$\gamma_1=
\begin{pmatrix}
\rho & 0 & 0 & \alpha \sigma^{-1}\\
0  & \rho^{-1} & \alpha \rho^{-1} & 0\\
0 & 0 & 1 & 0\\
0 & 0 & 0 & 1
\end{pmatrix},\;
\gamma_2=
\begin{pmatrix}
1 & 0 & 0 & 0 \\ 
0 & 1 & 0 & 0 \\
\alpha & 0 & \rho & 0 \\
0 & \alpha(\rho\sigma)^{-1} & 0 & \rho^{-1}
\end{pmatrix}, \; \alpha=(\rho+1)(\sigma+1).$$
 
We claim that $\Gamma$ is conjugate in $\GL_4(q^2)$ to $\SL_4(q)$.
To start, let $U\leq \F^{4}$ be a subspace fixed by $\Gamma$.
Let $\{f_1,f_2,f_3,f_4\}$ be the canonical basis of $\F^4$.
If $\gamma_{1|U}\neq \id$, then $f_i\in U$ for some $i=1,2$.
If follows $\langle f_i, \gamma_2 f_i, \gamma_1\gamma_2 f_i, \gamma_2\gamma_1\gamma_2 f_i\rangle=\F^4=U$.
Otherwise, calling $\overline{U}$ a complement of $U$ fixed by $\Gamma^T$, the subspace $\overline{U}$ contains the eigenvector 
$\overline{u}=\sigma f_1+(\sigma+1) f_4$ of $\gamma_1^T$.
Since $\langle \overline{u}, \gamma_2^T \overline{u},(\gamma_2\gamma_1\gamma_2)^T \overline{u}, 
(\gamma_2\gamma_1\gamma_2^2)^T\overline{u} \rangle =\F^4$, we conclude $U=\{0\}$.
Thus, $\Gamma$ is absolutely irreducible and we consider the maximal subgroups of $\SL_4(q)$ (see \cite[Tables 8.8
and 8.9]{Ho}).

Suppose that $\Gamma$ is imprimitive. Being irreducible, it maps onto a transitive subgroup of $\Sym(4)$.
From $\langle \gamma_1 \rangle=\langle \gamma_1^2\rangle$ and $\langle \gamma_2 \rangle=\langle \gamma_2^2\rangle$,
it follows that $\Gamma$ induces $\Alt(4)$ and the order of $g^6$ divides $(q-1)^4$ for all $g \in \Gamma$.
This implies $q=8$, but, in this case, $(\gamma_1\gamma_2^2)^6$ has order $65$.

By what observed at the beginning, $\F_2[\tr(\gamma_1)]=\F_q$.
Since the unique scalar matrix contained in $\SL_4(\F)$ is the identity, $\Gamma$ cannot be conjugate to any 
subgroup of $\SL_4(q_0)$ with $q_0<q$. Finally, direct computations show that $\Gamma$ fixes no symplectic or unitary forms.
Our claim for $\Gamma$ is proved. 

It follows that $\langle \tau,\tau^y\rangle_{|S_7}$ is conjugate to the group $\diag(\SL_4(q),\I_3)$,
as it fixes pointwise a $3$-dimensional space. In particular,
$\langle \tau,\tau^y\rangle_{|S_7}$ is generated by subgroups of root type.
Let $L$ be the normal closure of $\langle \tau,\tau^y\rangle_{|S_7}$ in $K_{|S_7}$.
Suppose there exists a non-trivial $L$-invariant subspace $U$ of $S_7$.
By Clifford's Theorem $U$ is $1$-dimensional, and $\langle \tau,\tau^y\rangle_{|S_7}$ is conjugate to 
a diagonal subgroup. This implies that $L$ is abelian, a contradiction.
As $q>2$, by  \cite[Theorem, page 364]{Mc}, we conclude that $L=\SL_7(q)=K_{|S_7}$. 
\end{proof}

\begin{corollary}\label{p210}
Suppose $p=2 < q$, and $n=10$ or  $n\geq 12$.
Then there exists $a \in \F_{q}^*$ such that $H=\Sp_{2n}(q)$.
\end{corollary}

\begin{proof}
Take $a\in \F_q^*$ as in \eqref{a sigma}. The subgroup $K=\langle G, G^x\rangle$ fixes $S_7$ inducing $\SL_7(q)$  by Lemma \ref{G7},
and so it is a subgroup $\K_7$: the statement follows from Theorem \ref{wSLell}.
\end{proof}

\subsection{Subcase $p$ odd, $n=10$ or $n\geq 12$} 

Take $\tau$ as in \eqref{deftau} and define 
\begin{equation*}
K= \left\langle  \tau^{y^i}, \tau^{yxy^j}, \tau^{(yx)^2y^k}\mid i,j,k\in \left\{0,\pm 1\right\}\right\rangle.
\end{equation*}
Recall that $\tau \in \wSL_9(q)$ and that, by the analysis of Section \ref{azione}, 
$\tau_{|V}$ is a transvection of center $u=e_{n-4} - 2a^{-1} e_{n-1} +  e_{n}$, such that
\begin{equation*}
\left\{\begin{array}{rcl}
\tau e_{n-7} & =& e_{n-7} + 4a^2 u,\\
\tau e_j  & = & e_j- 4a^2 u,\quad j \in \{n-3,n-2\},\\
\tau e_j & = & e_j, \quad j \in \{1,\ldots,n \}\setminus\{n-7,n-3,n-2 \}.
\end{array}\right. 
\end{equation*}
Note that $(yx)^{-1}S_9\subseteq S_{10}=\langle S_9, e_{n-9}\rangle$.
From $\tau^{yx} e_{n-9}=e_{n-9} +4a^2 u_4 $ with $u_4=(yx)^{-1}u \in S_9$ we get $\tau^{yx}S_9=S_9$.
Now, suppose $n= 13$ or $n\geq 15$. Then $(yx)^{-2} S_9\subseteq S_{11}=\langle S_{9}, e_{n-10},e_{n-9}\rangle$.
From $\tau^{(yx)^2} e_{n-9}=e_{n-9}$ and $\tau^{(yx)^2} e_{n-10}=e_{n-10} + 4a^2 u_7 $ with 
$u_7=(yx)^{-2}\in S_9$ we get $\tau^{(yx)^2}S_9=S_9$.
As $y$ fixes $S_9$ it follows that $K$ fixes $S_9$, inducing the identity on $\frac{W_9}{S_9}$, i.e., $K$ satisfies
condition $(\mathrm{b}_9)$ of Definition \ref{Kell}.
The same conclusion can be reached also when $n\in \{10,12,14\}$.

The generators of $K$ are bireflections, being conjugates of $\tau$, and their restrictions to $V$ are transvections.
The center of $(\tau_{|V})^T$ is $\overline u= e_{n-7}-e_{n-3}-e_{n-2}$.

\begin{lemma}\label{G9}
Suppose that  $a\in \F_q^\ast$ is such that:
$$\mathrm{(a)}\quad \F_p[a^3]=\F_q;\qquad  \mathrm{(b)}\quad a^{12}-14 a^9+92 a^6-312 a^3 +512\neq 0;
\qquad \mathrm{(c)}\quad (a^3+2)(a^6-27)\neq 0.$$
Then $K_{|S_9}=\SL_9(q)$.
\end{lemma}

\begin{proof}
By condition (b) the images of $u$ under the generators of $K$ are linearly independent. Hence
they generate $S_9$. By condition (c) also the images of $\overline u$
under the generators of $(K_{|S_9})^T$ are linearly independent.

Set $u_1=u$, $u_2=y^{-1}u$ and $u_3= yu$.
The group $G$ in \eqref{S9} fixes $U_3$ and, with respect to the basis
$\{8a u_1, u_2, a u_1 + a^3  u_2 + a^2 u_3\}$, its generators $\tau, \tau^y,\tau^{y^2}$ act  as
$$\begin{pmatrix}
1&1&0\\
0&1&0\\
0&0&1
\end{pmatrix},\ \begin{pmatrix}
1&0&0\\
-64a^3&1&0\\
0&0&1
\end{pmatrix},\ \begin{pmatrix}
-7& 1 & a^3-1\\
-64a^3 & 1+8a^3 & 8a^3(a^3-1)\\
64 &  -8   &  9-8a^3
\end{pmatrix}.$$
For  $q\ne 9$, by Dickson's Lemma and assumption (a), the group
$\langle \tau, \tau^y\rangle $ induces $\SL_2(q)$ on $U_2=\left\langle 8a u_1, u_2\right\rangle$.
Similarly, the group  $\left\langle \tau^T, \left(\tau^{y^2}\right)^T\right\rangle$
fixes $\overline{U}_2=\langle 8a \overline{u}, (y^2)^T \overline{u}  \rangle$
acting as  $\left\langle\begin{pmatrix} 1 & 1 \\ 0 & 1\end{pmatrix}, \begin{pmatrix} 1 & 0 \\ -64a^3 & 1\end{pmatrix}\right\rangle$, hence inducing $\SL_2(q)$.
For $q=9$, take $\overline{U}_3=\langle \overline{U}_2,  a \overline{u} + a^2 y^T \overline{u} + a^3 (y^2)^T\overline{u} \rangle$:
direct calculations give that both the restrictions $G_{|U_3}$ and $(G^T)_{|\overline U_3}$ are isomorphic to $\SL_3(9)$.
Set $G_2=\langle \tau, \tau^y \rangle$ when $q\neq 9$, and call $G_2$ the stabilizer of $U_2$ and $a u_1 + a^3  u_2 + a^2 u_3$ in $G$
when $q=9$.
The fixed point space $T_7$  of $G_2$ in $S_9$ has dimensions $7$ in both cases. Thus $G_{2|S_9}$ is generated
by subgroups of root type.
With respect to the decomposition $S_9=U_2\oplus T_7$ the element $z=\diag(-\I_2,\I_7)$ is in $G_{2|S_9}$.

Now ,suppose  that $U$ is a $K$-invariant subspace of $\F^9$.
If $U\cap U_2\ne \{0\}$,  then $U_2\le U$ by the irreducibility of
$\SL_2(q)$. From $u\in U$  we get $U=\F^9$, by what observed at the beginning.
If $U\cap U_2=\{0\}$, then $z_{|U}$ does not have
the eigenvalue $-1$. Calling $\overline U$ a complement of $U$ which is $\left(K_{|S_9}\right)^T$-invariant,
we  have  $\overline{U} \cap \overline{U}_2\neq \{0\}$, whence $\overline{U}=\F^9$ and  $U=\{0\}$.
It follows that $K_{|S_9}$ is irreducible: by Lemma \ref{McLaughlin} with $B=G_2$ we obtain $K_{|S_9}=\SL_9(q)$.
\end{proof}

\begin{corollary}\label{pno213}
Suppose that $p$ is odd,  and $n= 10$ or $n\geq 12$.
Then there exists $a \in \F_{q}^*$ such that $H=\Sp_{2n}(q)$.
\end{corollary}

\begin{proof}
We first show that  there exists $a\in \F_q^*$ which satisfies the 
assumptions of Lemma \ref{G9}.
If $q=p$, we have to exclude at most $21$ values of $a$.
So, if $p\geq 23$ we are done. If $p=3,5,7$ take $a=11$; if $11\leq p \leq 19$ take $a=-2$.
Suppose now that $q=p^f$, with $f\geq 2$, and let $\N(q)$ be the number of elements $b\in \F_q^*$ 
such that $\F_p[b^3]\neq 
\F_q$. We have to check when $p^f-1-\N(q) > 18$ (the condition $a^3\neq -2$ can be now dropped).
By Lemma \ref{campoN} we have
$\N(q)\leq 3p\frac{ p^{\left \lfloor f/2\right \rfloor} -1}{p-1}$
and hence it suffices to show that
$p^f- 3p\frac{ p^{\left \lfloor f/2\right \rfloor} -1}{p-1} > 19$.
This condition is fulfilled, unless $q=3^2,5^2,3^3$.
If $q=9,25$ take $a$ whose minimal polynomial over $\F_p$ is $t^2-2$.
If $q=27$, take $a$ whose minimal polynomial over $\F_3$ is $t^3-t+1$.

Now, by Lemma \ref{G9}, the subgroup $K$ fixes $S_9$ inducing $\SL_9(q)$ and so it is a subgroup $\K_9$: the statement follows from Theorem \ref{wSLell}.
\end{proof}

\section{Cases $n\in \{7,9,11\}$ and $q>2$}\label{caso n=7,9,11}

In this section we prove the $(2,3)$-generation of $\Sp_{2n}(q)$ when $n \in \{7,9,11\}$,
adapting to these cases the computations of Section \ref{sec4}.

\subsection{Case $n=7$ and $q>2$}

The element $\tau=[x,y]^{8}$ belongs to $N \rtimes \wSL_7(q)$ and its characteristic polynomial
is 
$$\chi_\tau(t)=\left\{\begin{array}{cl}
   (t + 1)^{10}(t^2 + a^8 t + 1)^2 & \textrm{ if } p=2,\\
(t-1)^{14} & \textrm{ if } p>2.
\end{array}\right.$$
The subgroup
$$K=\left\langle \tau^{y^i}, \tau^{y x}, \tau^{y^2[x,y]^4 y^j} \mid 
i,j=0,\pm 1  \right\rangle$$
fixes $V$, namely it is contained in $N\rtimes \wSL_7(q)$.

We start with the case $p=2$. As done in Section \ref{sec4}, if $q\neq 8$ we take $\sigma\in \F_{q^2}^*$ of order 
$q+1$ and define $a=\sigma+\sigma^{-1}$. In this way, $\F_2[a^3+a]=\F_q$.
It will be useful to observe that the following polynomials $f_i(t)\in \F_2[t]$ do not have $\sigma$ as root:
$$f_1(t)=t^3+t+1, \quad f_2(t)=t^3+t^2+1,\quad f_3(t)=t^5+t^2+1,\quad f_4(t)=t^5+t^3+1.$$
The same holds also for  $f_5(t)=t^8+t^7+t^6+t^4+t^2+t+1$, unless $q=16$.
So, for this value of $q$, we require that $\sigma$ has minimal polynomial $t^8+t^5+t^4+t^3+1$ over $\F_2$.
Finally, if $q=8$ take $a \in \F_8^*$ whose minimal polynomial over $\F_2$ is $t^3+t+1$.

\begin{lemma}\label{G72}
Assume $p=2 < q$ and take $a\in \F_q^*$ as previously described.
Then, the restriction of $K$ to $V$ is the group $\SL_{7}(q)$.
\end{lemma}

\begin{proof}
If $q=8$, let $g_1,g_2,g_3,g_4$ be the restrictions to $V$ of $\tau, \tau^{y^2}, \tau^{yx}, \tau^{y^2[x,y]^4y}$, respectively.
Then 
$$\bigcup_{k \in L}  \pi( (g_3 g_1 g_2  )^k g_4   ) =\pi(\SL_7(8)), \quad \textrm{where }
 L= \{1,2,4,6\}.$$
By Lemma \ref{LPS} we obtain $K_{|V}=\SL_7(8)$.
So, assume $q\neq 8$.

We first prove that the group $K_{|V}$ is absolutely irreducible.
The characteristic polynomial of $\tau_{|V}$ is 
$(t + 1)^5 (t^2 + a^8 t + 1)$  and  $\tau_{|V}$ has a fixed point space of dimension $5$.
Let $\rho=\sigma^{8}$. The  eigenspaces of $\tau_{|V}$ 
and 
$(\tau_{|V})^T$ relative to the eigenvalue $\rho^{\pm 1}$ are generated, respectively, by 
the following vectors, where $\alpha_1(\sigma)=\frac{\sigma^3+1}{\sigma (\sigma+1)^{4}}$ and
$\alpha_2(\sigma)=\sigma^3\alpha_1(\sigma)$:
\begin{equation}\label{sp}
\begin{array}{rcl}
s_{\rho^{\pm 1}} & = &e_{n-4}+ (\sigma^{\pm 1} + 1)  e_{n-1}+ e_{n},\\
\overline s_{\rho^{\pm 1}} & = & e_{n-4}+ \alpha_1(\sigma^{\pm 1}) e_{n-3} + \alpha_2(\sigma^{\pm 1}) e_{n-2} +e_{n-1}+ \sigma^{\mp 1} e_n.
\end{array}
\end{equation}
Observe that $\langle x s_{\rho^{\pm 1}}\rangle=\langle s_{\rho^{\mp 1}}\rangle$ and
$\langle x^T \overline s_{\rho^{\pm 1}} \rangle=\langle \overline s_{\rho^{\mp 1}}\rangle$.
Take the  square matrices $M_\pm\in \Mat_7(\F)$ whose columns are the restrictions to $V$  of the following vectors:
$$M_\pm : \quad  s_{\rho^{\pm 1}},\;   \tau^y s_{\rho^{\pm 1}}, \;   \tau^{y^2} s_{\rho^{\pm 1}},\; 
 \tau^{yx} s_{\rho^{\pm 1}} ,\; \tau^{y^2[x,y]^4} s_{\rho^{\pm 1}} , \;  
 s_{\rho^{\mp 1}}, \; \tau^{y^2[x,y]^4}  s_{\rho^{\mp 1}}.$$
We obtain that $\det(M_\pm)=a^{22} \sigma^{\mp 14} (f_3(\sigma^{\pm 1}))^2 (f_4(\sigma^{\pm 1}))^3 \neq 0$.
Now, let $\overline{M}_\pm$ be the matrix  whose columns are the restrictions to $V$  of the following vectors: 
$$\overline{M}_\pm : \quad  \overline s_{\rho^{\pm 1}},\;   (\tau^y)^T  \overline s_{\rho^{\pm 1}}, \;   
(\tau^{y^2})^T \overline s_{\rho^{\pm 1}},\;  (\tau^{y^2[x,y]^4y})^T \overline s_{\rho^{\pm 1}} ,\; 
\overline s_{\rho^{\mp 1}}, \; (\tau^{y^2})^T  \overline s_{\rho^{\mp 1}}, \;    (\tau^{y^2[x,y]^4y})^T \overline s_{\rho^{\mp 1}}.$$
Then $\det(\overline{M}_\pm)=
a^{18} (a+1)^5 \sigma^{\mp 15} (\sigma^{\pm 1} + 1)f_1(\sigma^{\pm 1}) 
(f_2(\sigma^{\pm 1}) f_5(\sigma^{\pm 1}))^2 \neq 0$: it follows that $M_\pm$ and $\overline{M}_\pm$ are invertible matrices.

Let $\{0\}\neq U$ be a $K_{|V}$-invariant subspace of $\F^7$ and
let $\overline{U}$ be a complement of $U$  which is $(K_{|V})^T$-invariant.
If $\tau_{|U}$ has both the eigenvalues $\rho, \rho^{- 1}$, then $s_{\rho},s_{\rho^{-1}}\in U$. Since 
$\det(M_+)\neq 0$, we obtain $U=\F^7$.
If  $\tau_{|U}$ has only one eigenvalue $\rho^{\pm 1}$,
then $\left( (\tau_{|V})^T\right)_{|\overline{U}} $ has the eigenvalue $\rho^{\mp 1}$. In this case, using the matrices $M_\pm $
and $\overline M_\mp$, we get
$\dim(U)\geq 4$ and $\dim(\overline{U})\geq 4$, a contradiction.
Finally, if $\tau_{|U}$ has only the eigenvalue $1$, then $(\tau_{|V})^T$ has both the eigenvalues 
$\rho, \rho^{-1}$ in $\overline{U}$. Since $\det(\overline M_+)\neq 0$, we get $\overline{U}=\F^7$ and so 
$U=\{0\}$. We conclude that $K_{|V}$ is absolutely irreducible.

Suppose now that $K_{|S_7}$ is imprimitive. This implies that the elements of $K_{|V}$  have orders that divide $7!(q-1)^7$.
Assume first $q\geq 16$. Then, there exists a  Zsigmondy prime $\ell\geq 11$ for $2^{f}+1$ (see \cite{Z}) 
that divides the order $\frac{q+1}{\gcd(q+1,3)}$ of $\tau_{|V}$. This clearly gives a contradiction.
If $q=4$, a similar contradiction can be obtained considering 
the element $\tau_{|V} (\tau^y)_{|V} ((\tau^{y^2})_{|V})^2$ which  has order $682$.

We can now apply \cite[Theorem 7.1]{GS}: $K_{|V}$ is a classical group in a natural representation.
Suppose there exists $g\in \GL_{7}(\F)$ such that $(K_{|V})^g\leq \GL_{7}(q_0)(\F^* \, \I_{7})$, where $\F_{q_0}$ is 
a subfield of $\F_q$. Set $(\tau_{|V})^g=\vartheta ^{-1}\tau_0$, with $\tau_0\in \GL_{7}(q_0)$ and  $\vartheta\in \F$.
Since $\tau_{|V}$ has the similarity invariant (i.e., the characteristic polynomial of a block
of the rational canonical form) $t+1$, the matrix $\tau_0=\vartheta(\tau_{|V})^g$ has the similarity 
invariant  $t+\vartheta$. It follows that $\vartheta \in \F_{q_0}$ and  the trace of $\tau_{|V}$ belongs to the subfield
$\F_{q_0}$. From  $\tr (\tau_{|V})=(a+1)^8$ we obtain $a\in \F_{q_0}$, and so
$\F_q=\F_2[a^3+a]\leq \F_2[a]\leq \F_{q_0}$.
We conclude that either $K_{|V}=\SL_{7}(q)$ or $K_{|V}\leq \SU_{7}(q_1^2) (\F_q^* \, \I_{7})$ with $q_1^2=q$.
Suppose the latest case occurs. 
Let $\psi: \GL_{7}(q)\to \GL_{7}(q)$ be the morphism defined by $(a_{i,j})^\psi=(a_{i,j}^{q_1})$. 
Then $(\tau_{|V})^\psi$ and $\lambda(\tau_{|V})^{-T}$ are conjugate elements for some $\lambda \in \F_q^*$.
Since the similarity invariants of $\tau_{|V}$ and $(\tau_{|V})^{-1}$ are
$p_1(t)=\ldots=p_4(t) = t + 1$ and
$p_5(t)= t^3 + (a + 1)^8 t^2 + (a + 1)^8 t + 1$,
we obtain $\lambda=1$ and $a^3+a \in \F_{q_1}$, 
whence the contradiction $\F_q=\F_2[a^3+a]\leq \F_{q_1}$.
This proves that $K_{|V}=\SL_{7}(q)$.
\end{proof}

We now consider the case $p$ odd.
The restriction of $\tau$ to $V$ is a transvection of
center $u=e_3-e_7$ such that
$$
\left\{\begin{array}{rclcrcl}
\tau e_j & =& e_j -2 (-1)^j a^2 u, \quad j \in \{4,5\},& \quad &
\tau e_6 & =& e_6+ 4 a u,\\
\tau e_j & = & e_j, \quad j \in \{1,2,3,7 \}.
\end{array}\right. 
$$
The generators of $K$ are bireflections, being conjugates of $\tau$, and their restrictions to $V$ are transvections.
The center of $(\tau_{|V})^T$ is $\overline{u}=e_4-e_5-2a^{-1}e_6$.

\begin{lemma}\label{K9}
Assume $p$ odd. Suppose that  $a\in \F_q^\ast$ is such that:
$$\mathrm{(a)}\quad\F_p[a^3]=\F_q; \qquad \mathrm{(b)}\quad a^3-8 \neq 0;\qquad \mathrm{(c)}\quad (a+1)(a^2-2)(a^2-2a+2)\neq 0. $$
Then $K_{|V}=\SL_7(q)$.
\end{lemma}

\begin{proof}
By condition (b) the images of $u$ under the generators of $K$ are linearly independent. Hence they generate $V$.
Set $\overline{u}_2 =((\tau^y)_{|V})^T\overline{u}-\overline{u}$.
By condition (c) also the set consisting of
$$(\tau_{|V})^T \overline{u},\;((\tau^y)_{|V})^T \overline{u},\;((\tau^{y^2})_{|V})^T \overline{u},\; ((\tau^{yx})_{|V})^T \overline{u},\; 
((\tau^{y^2[x,y]^4})_{|V})^T \overline{u},$$
$$((\tau^{y^2[x,y]^4y})_{|V})^T \overline{u}_2,\; ((\tau^{y^2[x,y]^4y^2})_{|V})^T \overline{u}_2$$
is linearly independent.

Set $u_1=u$, $u_2=y^{-1}u$ and $u_3=yu$. The group $G$ in \eqref{S9} fixes $U_3=\langle u_1,u_2,u_3\rangle$ and, with respect to the
basis $\{4a^2 u_1, -u_3, a^2u_1 -au_2+u_3\}$, its generators $\tau, \tau^y, \tau^{y^2}$ act as
\begin{equation*}
\begin{pmatrix}
1&1&0\\
0&1&0\\
0&0&1
\end{pmatrix},\ \begin{pmatrix}
1-4a^3 & 1 & -(a^3+1)\\
16a^3 & -3 & 4(a^3+1)\\
16a^3 & -4 & 1+4(a^3+1)
\end{pmatrix},\ \begin{pmatrix}
1 & 0 & 0\\
16a^3 & 1 & 0\\
0 & 0 & 1
\end{pmatrix}.
\end{equation*}
For $q\neq 9$, by Dickson's Lemma and  assumption (a), the group $\left\langle \tau , \tau^{y^2}\right\rangle$
induces  $\SL_2(q)$ on $U_2=\langle u_1,u_3\rangle$.
Similarly, the group $\left\langle (\tau_{|V})^T, ((\tau^{y})_{|V})^T\right\rangle$
fixes $\overline{U}_2=\langle \overline{u}, \frac{1}{16a^3} \overline{u}_2 \rangle$
acting as $\left\langle\begin{pmatrix}    1 & 1 \\    0 & 1\end{pmatrix}, 
\begin{pmatrix} 1 & 0 \\ 16a^3 & 1 \end{pmatrix}\right\rangle$, hence inducing $\SL_2(q)$.
For $q=9$, take $\overline{U}_3=\langle \overline{U}_2, \frac{a^3-1}{a^3} \overline{u} - \frac{1}{a^3} \overline{u}_2+
\frac{1}{a^3} ((\tau^{y^2})_{|V})^T \overline{u} \rangle$: direct calculations give that both the restrictions
$G_{|U_3}$ and $((G_{|V})^T)_{|\overline{U}_3}$ are isomorphic to $\SL_3(9)$.
Set $G_2=\langle \tau, \tau^{y^2}\rangle$ when $q\neq 9$, and call $G_2$ the stabilizer of $U_2$ and $a^2u_1 -au_2+u_3$ in $G$ when $q=9$.
The fixed point space $T_5$ of $G_2$ in $V$ has dimension $5$ in both cases. Thus $G_{2|V}$ is generated by subgroups of root type.
With respect to the decomposition $V=U_2\oplus T_5$, the element $z=\diag(-\I_2,\I_5)$ is in $G_{2|V}$.

Now suppose that $U\neq \{0\}$ is a $K$-invariant subspace of $\F^7$.
If  $U\cap U_2\neq \{0\}$, then $U_2\leq U$ by the irreducibility of $\SL_2(q)$. From $u\in U$ we get $U=\F^7$, by what observed at the beginning.
If $U\cap U_2=\{0\}$, then $z_{|U}$ does not have the eigenvalue $-1$. 
Calling $\overline U$ a complement of $U$ which is $\left(K_{|V}\right)^T$-invariant, 
we must have $\overline{U}\cap \overline{U}_2\neq \{0\}$, whence $\overline{U}=\F^7$ and $U=\{0\}$.
Thus, $K_{|V}$ is irreducible: by Lemma \ref{McLaughlin} with $B=G_2$ we obtain  $K_{|V}=\SL_7(q)$.
\end{proof}

\begin{lemma}\label{7ex}
If $q$ is odd, there exists $a \in \F_q^*$ satisfying all the hypotheses of {\rm Lemma \ref{K9}}.
\end{lemma}

\begin{proof}
If $q=p$, we have to exclude at most $8$ values of $a$.
So, if $p\geq 11$ we are done; if $p \in \{3,5\}$, take $a=1$; if $p=7$, take $a=5$.
Suppose now  $q=p^f$, with $f\geq 2$, and let $\N(q)$ be the number of elements $b\in \F_q^*$ 
such that $\F_p[b^3]\neq \F_q$. We have to check when $p^f-1-\N(q) > 4$,
as the conditions on $a$ of shape $a^3\neq c$ can be dropped.
By Lemma \ref{campoN}, it suffices to check when
$p^f - 3 p\frac{ p^{\left \lfloor f/2\right \rfloor} -1}{p-1} > 5 $.
This condition is fulfilled, unless $q=3^2$.
If $q=9$, take  $a\in\F_9$ whose minimal polynomial over $\F_3$ is $t^2-t-1$.
\end{proof}

\begin{proposition}\label{main7}
Suppose $q>2$. Then there exists $a \in \F_q^*$ such that $H=\Sp_{14}(q)$. 
\end{proposition}

\begin{proof}
Suppose first $p=2$. If $q \not\in \{8,16\}$, take $a=\sigma+\sigma^{-1}$, where $\sigma\in \F_{q^2}^*$ has order $q+1$.
If $q=2^f$ with $f \in\{3,4\}$, take $a \in \F_q$ whose minimal polynomial over $\F_2$ is $t^f+t+1$.
If $p$ is odd, take $a \in \F_q^*$ that satisfies the hypotheses of Lemma \ref{K9}
and whose existence follows from Lemma \ref{7ex}.
Then $H$ contains a subgroup $K$ of $N\rtimes \SL_7(q)$ which induces $\SL_7(q)$ on $V$.
By Theorem \ref{Guralnick-Saxl} we  conclude that  $H=\Sp_{14}(q)$.
\end{proof}

\subsection{Case $n=9$ and $q>2$}

The element $\tau=[x,y]^{12}$ belongs to $N \rtimes \wSL_9(q)$ and its characteristic polynomial
is 
$$\chi_\tau(t)=\left\{\begin{array}{cl}
 ( t + 1)^{14} (t^2 + (a^3 + a)^4 t + 1)^2  & \textrm{ if } p=2,\\
(t-1)^{18} & \textrm{ if } p>2.
\end{array}\right.$$
The subgroup
$$K=\left\{\begin{array}{ll}
\left\langle \tau,\tau^y, \tau^{y^2},  \tau^{y^2[x,y]^6 y^2 x} \right\rangle & \textrm{if } p=2,\\[5pt]
\left\langle \tau^{y^i}, \tau^{yx}, \tau^{y^2[x,y]^6 y^j}\mid i,j=0,\pm 1  \right\rangle & \textrm{if } p>2
\end{array}\right.$$
fixes $V$, and so it is contained in $N\rtimes \wSL_9(q)$. 
Direct computation shows that  $K$ fixes $S_7$ and induces the identity on $\frac{W_7}{S_7}$ (so it satisfies
condition $(\mathrm{b}_7)$ of Definition \ref{Kell}). We now prove that $K$ induces $\SL_7(q)$ on $S_7$.

We start with the case $p=2$. 
As done in Section \ref{sec4}, if $q\neq 8$ we take $\sigma\in \F_{q^2}^*$ of order 
$q+1$ and define $a=\sigma+\sigma^{-1}$. In this way, $\F_2[a^3+a]=\F_q$.
It will be useful to observe that the following polynomials $f_i(t)\in \F_2[t]$ do not have $\sigma$ as root:
$$ f_1(t)=t^5+t^4+t^3+t^2+1,$$
$$f_2(t)=t^{31}+t^{30}+t^{29}+t^{28}+t^{27}+t^{25}+t^{23}+t^{18}+t^{15}+t^{13}+t^9+t^7+t^4+t^2+1.$$
The same holds also for $f_3(t)=t^{12}+t^{11}+t^9+t^7+t^6+t^5+t^3+t+1$, unless $q=64$.
So, for this value of $q$, we require that $\sigma$ has minimal polynomial $t^{12}+t^8+t^7+t^6+t^5+t^4+1$ over $\F_2$.
Next, we observe that $\sigma^5+1\neq 0$ unless $q=4$.
Finally, if $q=8$ take $a \in \F_8^*$ whose minimal polynomial over $\F_2$ is $t^3+t+1$.

\begin{lemma}\label{K9even}
Assume $p=2 < q$ and take $a\in \F_q^*$ as previously described.
Then, the restriction of $K$ to $S_7$ is the group $\SL_{7}(q)$.
\end{lemma}

\begin{proof}
If $q\in \{4,8\}$, let $g_1,g_2,g_3,g_4$ be the respective restrictions of $\tau, \tau^{y}, \tau^{y^2}, \tau^{y^2[x,y]^6y^2x}$ to $S_7$.
Take  $L_4=\{1,5,7,11\}$ and $L_8=\{3,6,8,31\}$: then
\begin{equation}\label{pi7}
\bigcup_{k \in L_q}  \pi( (g_3 g_1 g_2  )^k g_4   ) =\pi(\SL_7(q)). 
\end{equation}
By Lemma \ref{LPS} we obtain $K_{|S_7}=\SL_7(q)$.
So, assume $q\neq 4,8$.

We first prove that the group $K_{|S_7}$ is absolutely irreducible.
The characteristic polynomial of $\tau_{|S_7}$ is 
$(t+1)^5(t^2 + (a^3 + a)^4 t + 1)$ and  $\tau_{|S_7}$ has a fixed point space of dimension $5$.
Let $\rho=\sigma^{12}$. Then $1\neq \rho\neq \rho^{-1}$; the  eigenspaces of $\tau_{|S_7}$ 
and 
$(\tau_{|S_7})^T$ relative to the eigenvalue $\rho^{\pm 1}$ are generated, respectively, by
the vectors $s_{\rho^{\pm 1}}$ and $\overline s_{\rho^{\pm 1}}$ as in \eqref{sp}, where
$\alpha_1(\sigma)=\frac{\sigma^{-1}}{\sigma^3 + 1}$ and $\alpha_2(\sigma)=\sigma^5 \alpha_1(\sigma)$.
Take the  square matrices $M_\pm\in \Mat_7(\F)$ whose columns are the restrictions to $S_7$  of the following vectors:
$$M_\pm : \quad  s_{\rho^{\pm 1}},\;   \tau^y s_{\rho^{\pm 1}}, \;   \tau^{y^2} s_{\rho^{\pm 1}},\; 
 \tau^{y^2[x,y]^6y^2x} s_{\rho^{\pm 1}} ,\;    s_{\rho^{\mp 1}},\;
\tau^y s_{\rho^{\mp 1}}, \; \tau^{y^2[x,y]^6y^2x}  s_{\rho^{\mp 1}}.$$
We obtain that $\det(M_\pm)=a^{13}(a+1)^{20}    \sigma^{\mp 12} (\sigma^{\pm 1} +1)
(\sigma^{\pm 5}+1)^3 f_1(\sigma^{\pm 1})\neq 0$. 
Let now $\overline{M}_\pm$ be the matrix obtained by $M_\pm$ replacing $s_{\rho^{\pm 1}}$ with $\overline s_{\rho^{\pm 
1}}$, $\tau^{y^{\pm 1}}$ with $(\tau^{y^{\mp 1}} )^T$, and $\tau^{y^2[x,y]^6y^2x}$ with its transpose.
We have $\det(\overline{M}_\pm)=a^7 (a+1)^9 \sigma^{\mp 31} (\sigma^{\pm 5}+1)^3f_2(\sigma^{\pm 1}) f_3(\sigma^{\pm 1})\neq 0$: 
it follows that $M_\pm$ and $\overline{M}_\pm$ 
are invertible matrices.

Now, we can argue as in the proof of Lemma \ref{G72}: this time,
the similarity invariants of $\tau_{|V}$ and $(\tau_{|V})^{-1}$ are
$p_1(t)=\ldots=p_4(t) = t + 1$ and
$p_5(t)= t^3 + (a^3+a+1)^4 t^2 + (a^3+a+1)^4 t + 1$.
\end{proof}

We now consider the case $p$ odd. 
The restriction of $\tau$ to $V$ is a transvection of
center $u= e_5 -2a^{-1} e_8 + e_9$ such that
$$
\left\{\begin{array}{rclcrcl}
\tau e_j & =& e_j + a^2 u, \quad j \in \{2,3,4\},& \quad &
\tau e_j & =& e_j - a^2 u, \quad j \in \{6,7\},\\
\tau e_j & = & e_j, \quad j \in \{1,5,8,9 \}.
\end{array}\right. 
$$
The generators of $K$ are bireflections, being conjugates of $\tau$, 
and their restrictions to $S_7$ are transvections.
The center of $(\tau_{|S_7})^T$ is $\overline{u}= e_3+e_4-e_6-e_7$.

\begin{lemma}\label{K9odd}
Assume $p$ odd. Suppose that  $a\in \F_q^\ast$ is such that:
\begin{itemize}
\item[\rm{(a)}] $\F_p[a^3(a+2)]=\F_q$;
\item[\rm{(b)}] $(a+2)(a^3-4)\neq 0$;
\item[\rm{(c)}] $( a^3-a+2) (a^4-a^3+6 a^2+6 a-4)\neq 0$.
\end{itemize}
Then $K_{|S_7}=\SL_7(q)$.
\end{lemma}

\begin{proof}
By condition (b) the images of $u$ under the generators of $K$ are linearly independent. Hence they generate $S_7$.
By conditions (b) and (c) also the images of $\overline{u}$ under the generators of $(K_{|S_7})^T$
are linearly independent.

Set $u_1=u$, $u_2=y^{-1}u$ and $u_3=yu$. The group $G$ in \eqref{S9} fixes $U_3$ and, with respect to the
basis $\left\{-2a^2u_1, u_3, u_1+ \frac{a+2}{2a} u_2+ \frac{(a+2)^2}{4a^2} u_3 \right\}$, its generators $\tau, \tau^y, \tau^{y^2}$ act as
$$\begin{pmatrix}
1 & 1 & 0\\
0 & 1 & 0\\
0 & 0 & 1
\end{pmatrix},\ \begin{pmatrix}
1+\frac{4a^3}{a + 2}  &     1                 &  \frac{-\beta}{4a^2(a+2)} \\
-2 a^3 (a+2)          &  1-\frac{(a+2)^2}{2}  &  \frac{(a+2)\beta}{8a^2} \\
 \frac{8 a^5}{a+2}    &    2a^2               &   1-\frac{\beta}{2(a+2)}
\end{pmatrix}, \ \begin{pmatrix}
1 & 0 &  0 \\
-2a^3(a+2) & 1 & 0 \\
0 & 0 & 1
\end{pmatrix},$$
where $\beta=(a-2)(7a^2+8a+4)$.
For $q\neq 9$, by Dickson's Lemma and  assumptions (a) and (b), the group $\left\langle \tau , \tau^{y^2}\right\rangle$
induces  $\SL_2(q)$ on $U_2=\langle u_1,u_3\rangle$.
Similarly, the group $\left\langle (\tau_{|S_7})^T, ((\tau^{y})_{|S_7})^T\right\rangle$
fixes $\overline{U}_2=\langle \overline{u}, 
\frac{-1}{2a^3(a+2)}\overline{u}_2  \rangle$,
where $\overline{u}_2= ((\tau^y)_{|S_7})^T \overline{u} -\overline{u}$,
acting as $\left\langle\begin{pmatrix}    1 & 1 \\    0 & 1\end{pmatrix}, 
\begin{pmatrix} 1 & 0 \\ -2a^3(a+2) & 1 \end{pmatrix}\right\rangle$ and hence inducing $\SL_2(q)$.
For $q=9$, take $\overline{U}_3=\langle \overline{U}_2, \overline{u}_3\rangle$,
where $\overline{u}_3=\frac{1}{a^2}\overline{u} + \frac{a+2}{a^5} ((\tau^y)_{|S_7})^T \overline{u}
-\frac{a+2}{a^5} ((\tau^{y^2})_{|S_7} )^T \overline{u}$: direct calculations give that both the restrictions
$G_{|U_3}$ and $((G_{|V})^T)_{|\overline{U}_3}$ are isomorphic to $\SL_3(9)$.
Set $G_2=\langle \tau, \tau^{y^2}\rangle$ when $q\neq 9$, and call $G_2$ the stabilizer of $U_2$ and
$u_1+ \frac{a+2}{2a} u_2+ \frac{(a+2)^2}{4a^2} u_3$ in $G$ when $q=9$.
Proceeding as done in the proof of Lemma \ref{K9} we conclude that $K_{|S_7}=\SL_7(q)$.
\end{proof}

\begin{lemma}\label{9ex}
If $q$ is odd, there exists $a \in \F_q^*$ satisfying all the hypotheses of {\rm Lemma \ref{K9odd}}.
\end{lemma}

\begin{proof}
If $q=p$, we have to exclude at most $11$ values of $a$.
So, if $p\geq 13$ we are done; 
if $p \in \{3,7,11\}$, take $a=-1$; if $p=5$, then take $a=1$.
Suppose now $q=p^f$, with $f\geq 2$, and let $\N(q)$ be the number of elements $b\in \F_q^*$ 
such that $\F_p[b^4+2b^3]\neq \F_q$. We have to check when $p^f-1-\N(q) > 10$,
as the conditions on $a$ of shape $a\neq c$ can be dropped.
By Lemma \ref{campoN}, it suffices to show that
$p^f - 4p\frac{ p^{\left \lfloor f/2\right \rfloor} -1}{p-1} > 11 $.
This condition is fulfilled, unless $q=3^2,5^2$.
If $q\in \{9,25\}$ take $a\in \F_q$ whose minimal polynomial over $\F_p$ is $t^2+7$.
\end{proof}

\begin{proposition}\label{main9}
Suppose $q>2$. Then there exists $a \in \F_q^*$ such that $H=\Sp_{18}(q)$. 
\end{proposition}

\begin{proof}
Suppose $p=2$. If $q\not \in \{8,64\}$, take $a=\sigma+\sigma^{-1}$, where $\sigma\in \F_{q^2}^*$ has order $q+1$;
if $q=2^f$ with $f \in \{3,6\}$, take $a \in \F_q^*$ whose minimal polynomial over $\F_2$ is $t^f+t+1$.
If $p$ is odd, take $a \in \F_q$ that satisfies the hypotheses of Lemma \ref{K9odd}
 and whose existence follows from Lemma \ref{9ex}.
As already observed, $K$ fixes $S_7$ and induces the identity on $\frac{W_7}{S_7}$.
By Lemmas \ref{K9even} and \ref{K9odd} it also induces $\SL_7(q)$ on $S_7$, so it is a subgroup $\K_7$.
By Theorem \ref{wSLell} it follows that $H=\Sp_{18}(q)$.
\end{proof}

\subsection{Case $n=11$ and $q>2$}

The element $\tau=[x,y]^{16}$ belongs to $N \rtimes \wSL_{11}(q)$ and its characteristic polynomial
is 
$$\chi_\tau(t)=\left\{\begin{array}{cl}
 ( t + 1)^{18} (t^2 +  a^{16} t + 1)^2  & \textrm{ if } p=2,\\
(t-1)^{22} & \textrm{ if } p>2.
\end{array}\right.$$
The subgroup
$$K=\left\{\begin{array}{ll}
\left\langle  \tau^{y^i}, \tau^{y^jx }  \mid i,j \in \{0,\pm 1\}  \right\rangle & \textrm{if } p=2,\\[5pt]
\left\langle  \tau^{y^i}, \tau^{yxy^j} \tau^{y[x,y]^6 y^k } \mid i,j,k \in \{0,\pm 1\}  \right\rangle & \textrm{if } p>2
\end{array}\right.
$$
fixes $V$, and so it is contained in $N\rtimes \wSL_{11}(q)$.

We start with the case $p=2$. 
As done in Section \ref{sec4}, if $q\neq 8$ we take $\sigma\in \F_{q^2}^*$ of order 
$q+1$ and define $a=\sigma+\sigma^{-1}$. In this way, $\F_2[a^3+a]=\F_q$.
It will be useful to observe that the following polynomials $f_i(t)\in \F_2[t]$ do not have $\sigma$ as root:
$$f_1(t)=t^5+t^3+t^2+t+1, \quad f_2(t)= t^5+t^4+t^2+t+1,\quad
f_3(t)=t^7+t+1,$$
$$f_4(t) =t^7+t^6+1,\quad f_5(t)=t^{11}+t^7+t^6+t^5+t^4+t^2+1,$$
$$f_6(t)=  t^{12}+ t^6+ t^3+ t^2+1,\quad  f_7(t)=t^{12}+t^{11}+t^8+t^5+t^4+t^2+1,$$ 
$$\quad f_8(t)=t^{18}+t^{17}+t^{15}+t^{12}+t^9+t^7+t^6+t^4+1.$$
The same holds also for $f_9(t)=t^{12}+t^{10}+t^7+t^6+t^5+t^2+1$, unless $q=64$.
So, for this value of $q$, we require that $\sigma$ has minimal polynomial $t^{12}+t^8+t^7+t^6+t^5+t^4+1$ over $\F_2$
Next, we observe that $(\sigma^5+1)(\sigma^7+1)\neq 0$ unless $q=4$.
Finally, if $q=8$ take $a \in \F_8^*$ whose minimal polynomial over $\F_2$ is $t^3+t^2+1$.

 Direct computation shows that  $K$ fixes $S_7$ and induces the identity on $\frac{W_7}{S_7}$ (so it satisfies
condition $(\mathrm{b}_7)$ of Definition \ref{Kell}). We now prove that $K$ induces $\SL_7(q)$ on $S_7$.

\begin{lemma}\label{G11}
Assume $p=2 < q$ and take $a\in \F_q^*$ as previously described.
Then, the restriction of $K$ to $S_7$ is the group $\SL_{7}(q)$.
\end{lemma}

\begin{proof} 
If $q\in \{4,8\}$, let $g_1,g_2,g_3,g_4$ be the restrictions to $S_7$ of $\tau^y, \tau^{yx}, \tau, \tau^{y^2}$, respectively.
Take  $L_4=\{1,2,3,4,8\}$ and $L_8=\{ 1,2,9,34\}$: then \eqref{pi7} holds and by Lemma \ref{LPS} we obtain $K_{|S_7}=\SL_7(q)$.
So, assume $q\neq 4,8$.

We first prove that the group $K_{|S_7}$ is absolutely irreducible.
The characteristic polynomial of $\tau_{|S_7}$ is 
$(t + 1)^5 (t^2 + a^{16} t + 1)$  and  $\tau_{|S_7}$ has a fixed point space of dimension $5$.
Let $\rho=\sigma^{16}$. Then $1\neq \rho\neq \rho^{-1}$. 
The  eigenspaces of $\tau_{|S_7}$ and $(\tau_{|S_7})^T$ relative to the eigenvalue $\rho^{\pm 1}$ are generated, respectively, by
the vectors $s_{\rho^{\pm 1}}$ as in \eqref{sp}  and 
$$\overline s_{\rho^{\pm 1}}=e_5+ \sigma^{\mp 1} e_6 + 
\frac{a^3\sigma^{\mp 1} (\sigma^{\pm 1}+1) }{a+1} (e_7+e_{10})+\sigma^{\mp 4} e_8 + \sigma^{\pm 3} e_9
+\frac{a^3\sigma^{\mp 2} (\sigma^{\pm 1}+1)}{a+1} e_{11}.$$
Observe that $\langle x s_{\rho^{\pm 1}}\rangle=\langle s_{\rho^{\mp 1}}\rangle$ and
$\langle x^T \overline s_{\rho^{\pm 1}} \rangle=\langle \overline s_{\rho^{\mp 1}}\rangle$.
Take the  square matrices $M_\pm\in \Mat_7(\F)$ whose columns are the restrictions to $S_7$  of the following vectors:
$$M_\pm : \quad  s_{\rho^{\pm 1}},\;   \tau^y s_{\rho^{\pm 1}}, \;   \tau^{y^2} s_{\rho^{\pm 1}},\; 
 \tau^{yx} s_{\rho^{\pm 1}} ,\;  \tau^{y^2x} s_{\rho^{\pm 1}} ,\;   s_{\rho^{\mp 1}}, \;
 \tau^{y}  s_{\rho^{\mp 1}}.$$
We obtain that $\det(M_\pm)=\sigma^{\mp 23 }  a^{41} (a+1)^3 (\sigma^{\pm 1}+1)^2   (\sigma^{\pm 5}+1) 
(\sigma^{\pm 7}+1)^2 
 f_1(\sigma^{\pm 1}) f_2(\sigma^{\pm1})$ $f_6(\sigma^{\pm 1})\neq 0$.
Let now $\overline{M}_\pm$ be the matrix  whose columns are the restrictions to $S_7$  of the following vectors: 
$$\overline{M}_\pm : \quad  \overline s_{\rho^{\pm 1}},\;   (\tau^y)^T  \overline s_{\rho^{\pm 1}}, \;   
(\tau^{y^2})^T \overline s_{\rho^{\pm 1}},\;  (\tau^{yx})^T \overline s_{\rho^{\pm 1}} ,\; 
 (\tau^{y^2x})^T \overline s_{\rho^{\pm 1}} ,\; 
\overline s_{\rho^{\mp 1}}, \; (\tau^{y})^T  \overline s_{\rho^{\mp 1}}.$$
We have $\det(\overline{M}_\pm)= \frac{a^{58}}{(a+1)^7} \sigma^{\mp 36} (\sigma^{\pm 1}+1)  (\sigma^{\pm 5}+1)
f_3(\sigma^{\pm 1})f_4(\sigma^{\pm 1}) f_5(\sigma^{\pm 1})f_7(\sigma^{\pm 1})f_8(\sigma^{\pm 1})$ $f_9(\sigma^{\pm 1})\neq 0$: 
it follows that $M_\pm$ and $\overline{M}_\pm$ are invertible matrices.

Now,  we can argue as in the proof of Lemma \ref{G72}: this time,
the similarity invariants of $\tau_{|V}$ and $(\tau_{|V})^{-1}$ are
$p_1(t)=\ldots=p_4(t) = t + 1$ and
$p_5(t)=t^3 + (a+1)^{16}  t^2 + (a+1)^{16} t + 1$.
\end{proof}

We now consider the case $p$ odd. The restriction of $\tau$ to $V$ is a transvection of
center $u=e_7-e_{11}$ such that
$$
\left\{\begin{array}{rclcrcl}
\tau e_{10}  & =& e_{10} + 8a u, &\quad &
\tau e_j  & = & e_j+ 4a^2 u,\quad j \in \{1,2,5,9\},\\
\tau e_j & = & e_j, \quad j \in \{ 7,11\},&&
\tau e_j & = & e_j- 4a^2 u, \quad j \in \{ 3,4,6,8\}.
\end{array}\right. 
$$
A direct computation shows that $K$ fixes the subspace $S_9$, acting as the identity on
$\frac{W_9}{S_9}$. The generators of $K$ are bireflections, being conjugates of $\tau$, 
and their restrictions to $S_9$ are transvections.
The center of $(\tau_{|S_9})^T$ is $\overline u  = a(-e_3-e_4+ e_5- e_6- e_8+ e_9)+2e_{10}$.

\begin{lemma}\label{Gn11} 
Assume $p$ odd. Suppose that  $a\in \F_q^\ast$ satisfies:
\begin{itemize}
\item[\rm{(a)}] $\F_q=\F_p[a^3(a+2)]$;
\item[\rm{(b)}] $(a+1)(a+2) (a^2+3a+3) (2a^2+6a-1)(6a^2 + 1)\neq 0$;
\item[\rm{(c)}] $(3a+4)(10a^5+10a^4-7a^3+13a^2+8a-8)
(37a^{10}-10a^9+260a^8+9a^7+7a^6+277a^5-91a^4-70a^3+84a^2-56a+16)\neq 0$.    
\end{itemize}
Then $K_{|S_9}$ is the group $\SL_{9}(q)$.
\end{lemma}

\begin{proof}
By condition (b) the images of $u$ under the generators of $K$ are linearly independent. Hence they generate $S_9$.
By conditions (b) and (c), the set $\Lambda \cup \{((\tau^{y[x,y]^6y})_{|S_9})^T((\tau^{y})_{|S_9})^T \overline{u} \}$, where 
$\Lambda$ consists of the images of $\overline{u}$ under the generators $g\neq ((\tau^{y[x,y]^6y})_{|S_9})^T$ of $(K_{|S_9})^T$,
is linearly independent.

Set $u_1=u$, $u_2=y^{-1}u$ and $u_3=yu$. The group $G$ in \eqref{S9} fixes $U_3$ and, with respect to the
basis $\left\{ u_1, -4 a (a+2) u_3,  au_1-\frac{a+2}{2} u_2  + \frac{(a+2)^2}{4a} u_3  \right\}$, 
its generators $\tau, \tau^y, \tau^{y^2}$ act as
$$\begin{pmatrix}
1 & 32a^3(a+2) &0\\
0 & 1& 0\\
0 & 0& 1
\end{pmatrix},\ \begin{pmatrix}
1-\frac{16a^3}{a+2} &  32a^3(a+2) &  \frac{-2a\beta}{a+2}\\
1   &  1-2(a+2)^2 &  \frac{\beta}{8a^2}\\
\frac{16a^2}{a+2} & -32a^2( a+2) & 1+\frac{2\beta}{a+2}
\end{pmatrix},\ 
\begin{pmatrix}
1 & 0 & 0 \\
1 & 1 & 0 \\
0 & 0 & 1
\end{pmatrix},$$
where $\beta=(3a+2)(3a^2+4)$.
For $q\neq 9$, by Dickson's Lemma and  assumptions (a) and (b), the group $\left\langle \tau , \tau^{y^2}\right\rangle$
induces  $\SL_2(q)$ on $U_2=\langle u_1,u_3\rangle$.
Similarly, the group $\left\langle (\tau_{|S_9})^T, ((\tau^{y})_{|S_9})^T\right\rangle$
fixes $\overline{U}_2=\langle \overline{u}, \overline{u}_2  \rangle$,
where $\overline{u}_2= ((\tau^y)_{|S_9})^T \overline{u} -\overline{u}$,
acting as $\left\langle\begin{pmatrix}    1 & 32a^3(a+2)  \\    0 & 1\end{pmatrix}, 
\begin{pmatrix} 1 & 0 \\ 1 & 1 \end{pmatrix}\right\rangle$ and hence inducing $\SL_2(q)$.
For $q=9$, take $\overline{U}_3=\langle \overline{U}_2, \overline{u}_3\rangle$,
where $\overline{u}_3=
 a\overline u -\frac{a+2}{a^2}  ((\tau^y)_{|S_9})^T \overline u +\frac{a+2}{a^2} ((\tau^{y^2})_{|S_9})^T
 \overline u$: direct calculations give that both the restrictions
$G_{|U_3}$ and $((G_{|V})^T)_{|\overline{U}_3}$ are isomorphic to $\SL_3(9)$.
Set $G_2=\langle \tau, \tau^{y^2}\rangle$ when $q\neq 9$, and call $G_2$ the stabilizer of $U_2$ and
$ au_1+(a+2)  u_2  + \frac{(a+2)^2}{a} u_3 $ in $G$ when $q=9$.
Proceeding as done in the proof of Lemma \ref{K9} we conclude that $K_{|S_9}=\SL_9(q)$.
\end{proof}

\begin{lemma}\label{11ex}
If $q \geq 5$ is odd, there exists $a \in \F_q^*$ satisfying all the hypotheses of {\rm Lemma~\ref{Gn11}}.
\end{lemma}

\begin{proof}
If $q=p\geq 5$, we have to exclude at most $24$ values of $a$.
So, if $p\geq 29$ we are done; 
if $13\leq p \leq 23$, take $a=-3$;
if $p=7$, take $a=2$;
if $p\in \{5,11\}$, take $a=1$.
Suppose now $q=p^f$, with $f\geq 2$, and let $\N(q)$ be the number of elements $b\in \F_q^*$ 
such that $\F_p[b^4+2b^3]\neq \F_q$. We have to check when $p^f-1-\N(q) > 21$,
as the conditions on $a$ of shape $a\neq c$ can be dropped.
By Lemma \ref{campoN}, we have
$\N(q)\leq 4 p\frac{ p^{\left \lfloor f/2\right \rfloor} -1}{p-1}$
and hence it suffices to show that
$p^f - 4p\frac{ p^{\left \lfloor f/2\right \rfloor} -1}{p-1} > 22$.
This condition is fulfilled, unless $q=3^2,3^3, 5^2, 7^2$.
If $q=9$, take  $a\in \F_9$ whose minimal polynomial over $\F_3$ is $t^2+t-1$.
If $q=25,49$ take $a\in \F_q$ whose minimal polynomial over $\F_p$ is $t^2+2$.
If $q=27$, take $a\in \F_{27}$ whose minimal polynomial over $\F_3$ is $t^3-t+1$.
\end{proof}

\begin{proposition}\label{main11}
Suppose $q>2$. Then there exists $a \in \F_q^*$ such that $H=\Sp_{22}(q)$. 
\end{proposition}

\begin{proof}
Suppose first $p=2$. If $q \not\in \{8,64\}$, take $a=\sigma+\sigma^{-1}$, where $\sigma\in \F_{q^2}^*$ has order $q+1$.
If $q=8,64$ take $a \in \F_q$ whose minimal polynomial over $\F_2$ is, respectively, $t^3+t^2+1$ and $t^6+t+1$.
As we already observed, $K$ fixes $S_7$ and induces the identity on $\frac{W_7}{S_7}$.
By Lemmas \ref{G11}  it  induces $\SL_7(q)$ on $S_7$, so it is a subgroup $\K_7$.
By Theorem \ref{wSLell} it follows that $H=\Sp_{22}(q)$.

Suppose $p$ is odd. If $q\geq 5$, take $a \in \F_q^*$ that satisfies the hypotheses of Lemma \ref{Gn11}
and whose existence follows from Lemma \ref{11ex}.
Then, $K$ fixes $S_9$ inducing $\SL_9(q)$.
Since $K$ acts as the identity on $\frac{W_9}{S_9}$, we obtain that 
$H$ contains a subgroup $\K_9$ and the statement follows again from Theorem \ref{wSLell}.

Finally, take $q=3$, $a=1$ and
$L=\{1,4,6,8,10,11,12, 16, 20, 28, 35\}$. Then 
$$\bigcup_{k \in L} \pi( (xy)^k y )  =\pi(\Sp_{22}(3)).$$
By Lemma \ref{LPS} we conclude that  $H=\Sp_{22}(3)$.
\end{proof}

\section{Case $n=6$ and $q>2$}\label{caso n=6}

\begin{lemma}\label{q4}
Let $q=4$ and $x,y$ be as in  {\rm Section \ref{sec2}}.
Then $H=\langle x,y\rangle$ is $\Sp_{12}(4)$.
\end{lemma}

\begin{proof}
We have
$$\bigcup_{k \in L} \pi( (xy)^k y )=\pi(\Sp_{12}(4)),\quad \textrm{where }  L= \{2,3,5,6,23,28\}.$$
By Lemma \ref{LPS} we conclude that either $H=\Sp_{12}(4)$ or 
$\Omega_{12}^-(4) \unlhd H$.
The second case can be excluded arguing as in Proposition \ref{q=2} for $n=6$.
Thus, $H=\Sp_{12}(4)$.
\end{proof}

Computational evidences suggest that the elements $x,y$ described in Section \ref{sec2}
work also for $n=6$. However, for  $q\neq 2,4$, we prefer to take different generators for a simpler proof.
So, define $x,y$ of respective order $2$ and $3$, as 
follows:
\begin{itemize}
\item $x$ swaps $e_{\pm 1}$ with $e_{\pm 2}$, and  swaps $e_{\pm 3}$ with $e_{\pm 4}$;
\item $x$ acts on $\langle e_5,e_6\rangle$ with matrix $\gamma$ as in \eqref{matGamma} and acts on $\langle e_{-5},e_{-6} \rangle$ with matrix $\gamma^T$;
\item $y e_1=e_3$, $ye_{3}=-(e_1+e_3)$, $ye_{-1}=-e_{-1}+e_{-3}$ and $ye_{-3}=-e_{-1}$;
\item $y e_2=e_{-2}$ and $ye_{-2}=-(e_2+e_{-2})$;
\item  $y$ acts on $\langle e_{\pm 4}, e_{\pm 5}, e_{\pm 6}\rangle$ as the permutation
$\left(e_{\pm 4 },e_{\pm 5},e_{\pm 6}\right)$.
\end{itemize}
Since $x^T J x  = y ^T J y = J$, it follows that $H=\langle x,y\rangle$ is contained in $\Sp_{12}(q)$.

A direct calculation gives that the characteristic polynomial $\chi_{[x,y]}(t)$ of $[x,y]$ is:
$$(t+1)^4(t^4 - t^3 + t^2 - t + 1)^2.$$
Moreover, when $a\neq 0$, the eigenspace $\V_{-1}([x,y])$ 
is generated by
$w_1$ and $w_2=xw_1$,  where
$$w_1=a (e_1+e_{-3})- (e_5+ e_{-5})\equad w_2= a (e_2-e_6+e_{-4})+ e_5+ e_{-5}.$$

\begin{lemma}\label{irr6} 
Assume $q\neq 4$. Let $a\in \F_q^*$ such that $\F_p[a]=\F_q$; 
if $p=11$,  suppose also $a^2+1\neq 0$.
Then the group $H$ is absolutely irreducible.
\end{lemma}

\begin{proof} 
Let $U$ be a proper $H$-invariant subspace of $\V=\F^{12}$.
Suppose that $W=\V_{-1}([x,y])$ is contained in $U$.
Define  $g_1= y[x,y]xy$, $g_2=xg_1$ and $g_3=g_2y$.
For $i=1,2,3$, consider the subspaces $T_i=W + yW + yxy^2 W+  (xy)^2 y W + y[x,y] x y^2 W+ g_i W$.
We show that at least one of the subspaces $T_1,T_2,T_3\leq U$ has dimension $12$.
In fact, let $M_i$ be the matrix whose columns are the vectors $w_1,w_2$ and their images by
$y, yxy^2, (xy)^2 y, y[x,y] x y^2, g_i$.
Then $\det(M_1)=-4 a^8 (a^{10} - 2 a^8 - 3 a^6 - 32 a^4 + 103 a^2 - 24)$,
$\det(M_2)=2 a^8 (a^{10}+6a^8-11a^6-110 a^4+229 a^2-50)$
and $\det(M_3)=-a^{10}(4a^2-1)(a^2+a-1)^2(a^2-a-1)^2$.
Clearly, if $(4a^2-1)(a^2+a-1)(a^2-a-1)\neq 0$ we have done.
In particular, this is true if $p=2$ since we are assuming $q\neq 4$.
So, suppose $p\neq 2$.
If $a^2=1\pm a$,  then $\det(M_2)= 16 a^8 (2\mp a)$.
This means that $\det(M_2)\neq 0$ unless $a=\pm 2$. However, from $a^2=1 \pm a$ we get the contradiction $4=3$.
If $a=\pm \frac{1}{2}$, we get  $\det(M_1)= 2^{-16}\cdot 311$ and $\det(M_2)=2^{-17}\cdot 233$. 
So, these determinants cannot be both zero.

Then, by Lemma \ref{HT}, we may suppose that $\dim(U)\leq 6$ and, since $xw_1=w_2$, that
$\V_{-1}([x,y])\cap U=\langle u_\pm\rangle$, where $u_\pm=w_1\pm w_2$.
The matrix $M_\pm$ whose columns are the images of $u_\pm$ under 
$\I_{12}, y, y^2 , xy^2, yxy^2, (xy^2)^2, y (xy^2)^2$ has rank $7$, as we are going to show.
In fact, the submatrix of $M_+$ corresponding to the rows $1, 2, 3, 5, 7, 11, 12$
has determinant  $-a^7\neq 0$. 
Suppose now that $p$ is odd.
The submatrix of $M_-$ corresponding to the rows $2, 3, 6, 7, 9, 11, 12$
has determinant  $-a^3(a+2)^2(a^6-a^4-8a^2+16)$.
If $a=-2$, the submatrix of $M_-$ corresponding to the rows $1, 2, 3, 4, 5, 6, 10 $
has determinant $2^{12}\neq 0$. So, we may consider the case $a\neq -2$ and $a^6-a^4-8a^2+16=0$, which implies $a\neq -1, 2$.
The submatrix of $M_-$ corresponding to the rows $1,2,3,4,5,6,12$
has determinant $-2a(a+1)^2(a+2)^2(a-2)^3(a^2+1)$, which is nonzero unless $a^2=-1$.
In this case, it follows that $0=a^6-a^4-8a^2+16=2\cdot 11$ and so $p=11$, which is excluded by our initial
hypotheses.
\end{proof}

Note that if $p=11$ and $a\in \F_{121}$ is such that $a^2=-1$, the subspace
generated by the images of $w_1-w_2$ under $\I_{12}, y, y^2, xy^2, yxy^2, [x,y^2]$
is $H$-invariant.

The characteristic polynomial of $\tau=[x,y]^5$ is $(t+1)^{12}$ and the eigenspace
$\Vb_{-1}(\tau)$
has dimension $10$.
It follows that $\tau$ is a bireflection and $[x,y]^{5p}=-\I_{12}$.
We will need the traces of the following elements:
\begin{equation}\label{trace6}
\tr(y)=-3, \quad \tr(xy)=a,\quad \tr([x,y])=-2,\quad \tr([x,y]xy)=-a.
\end{equation}

\begin{lemma}\label{prim6} 
Suppose that $H$ is absolutely irreducible.
Then, $H$ is primitive and tensor indecomposable.
\end{lemma}

\begin{proof}
Suppose first that $H$ preserves a decomposition  $\Vb=\F_q^{12}=V_1\oplus \ldots \oplus V_\ell$, $\ell >1$. 
By the irreducibility, the permutation action must be transitive. Also,
by \cite[Table 3.5.C]{KL} it suffices to consider the cases $\ell\in \left\{ 2,3,6\right\}$.\\
\noindent \textbf{Case $\ell=2$.}
We must have $yV_1=V_1$, $yV_2=V_2$. It follows $\tr (xy)=0$ in contrast with  \eqref{trace6}.\\
\noindent \textbf{Case $\ell=3$.} By the transitivity of the action, we may assume $V_2=yV_1$ and $V_3=yV_2$.
Since $\tr(y)=-3$, this implies that $p= 3$.
We may also assume $xV_1=V_1$ and either (1) $xV_2=V_2$ and $xV_3=V_3$ or (2) $xV_2=V_3$.
In case (1) we get again  the contradiction $\tr (xy)=0$; 
in case (2), we get the contradiction $\tr([x,y])=0$, see \eqref{trace6}.\\
\noindent \textbf{Case $\ell=6$.}
By the transitivity of the action, we may assume $V_2=yV_1$ and $V_3=yV_2$.
If $yV_i=V_i$ for all $i\geq 4$,  we may assume $xV_1=V_4$, $xV_2=V_5$ and $xV_3=V_6$.
In this case, it follows that $\tr(xy)=0$, once again a contradiction with \eqref{trace6}.
So, we may also assume $V_4=xV_3$, $V_5=yV_4$ and $V_6=yV_5$.
Since $\tr(y)=-3$, this means that $p=3$.
Now, $\tr(xy)$ and $\tr([x,y])\neq 0$ imply that $xy$ and $[x,y]$ fix at least a subspace $V_i$:
either (1) $xV_1=V_2$,  $xV_3=V_4$ and $xV_i=V_i$ for $i=5,6$; 
or (2) $xV_3=V_4$, $xV_5=V_6$ and $xV_i=V_i$ for $i=1,2$.
In both cases, we obtain a contradiction with $\tr([x,y]xy)=-a\neq 0$.
\smallskip

Assume next that $H$ is conjugate to a subgroup of $\GL_2(\F)\otimes \GL_6(\F)$
and let $\diag(J_2,J_6)$ be a preimage of $[x,y]^5$ in $\GL_2(\F)\times \GL_6(\F)$.
Up to conjugation we may assume 
that $J_2$ and $J_6$ are  upper triangular Jordan canonical forms. Suppose that $J_2$ is scalar. 
Then the normal closure of $[x,y]^5$ under $H$
would be reducible and, by Clifford's Theorem, $H$ would be imprimitive, a contradiction. By the same reason $J_6$ cannot 
be scalar.
Noting that the Kronecker product of upper triangular matrices is upper triangular, since $[x,y]^5$ has only 
the eigenvalue $-1$, $J_2$ and $J_6$ must have unique eigenvalues $\alpha$ and $-\alpha^{-1}$, respectively. Thus  
$J_2=\begin{pmatrix} 
\alpha&1\\
0&\alpha
\end{pmatrix}$ and $J_6$ is a sequence of at most five Jordan blocks with common eigenvalue $-\alpha^{-1}$. In each 
of the 
possible cases, the eigenspace relative to the eigenvalue $-1$ has always dimension at most $6<10$, a contradiction.

Finally, assume that $H$ is conjugate to a subgroup of $\GL_4(\F)\otimes \GL_3(\F)$.
Repeating the previous argument we obtain eigenspaces relative to the eigenvalue $-1$ of dimension at most
$7< 10$. We conclude that $H$ is tensor indecomposable.
\end{proof}

\begin{theorem}\label{main6}
Suppose $q\neq 2,4$ and let $a \in \F_q^*$ such that $\F_q=\F_p[a^2]$.
Then $H=\Sp_{12}(q)$. 
In particular $\Sp_{12}(q)$ is $(2,3)$-generated for all $q>2$.
\end{theorem}  

\begin{proof} 
Clearly, we can take an element  $a\in \F_q^*$ that satisfies our hypothesis
(if $q=11^2$, by  $\F_{121}=\F_{11}[a^2]$ we get $a^2\neq -1$).
By Lemmas \ref{irr6} and \ref{prim6}
the group $H$ is absolutely irreducible, primitive and tensor indecomposable. Moreover $H$  contains the 
bireflection $[x,y]^5$: by \cite[Theorem 7.1]{GS} $H$  is a classical group in a natural representation.

So, suppose that $H$ is contained in a maximal subgroup of class $\mathcal{C}_5$. 
Then, for every  $h \in H$  we have $h^2\in \Sp_{12}(\F_{q_0})$ for some subfield
$\F_{q_0}$ of $\F_{q}$. In particular, we get
$a^2 + 2 =\tr ((xy)^2) \in \F_{q_0}$. We conclude $q_0=q$ 
by the assumption $\F_q= \F_p[a^2]$.

Finally, suppose that $p=2$ and $H$ preserves a quadratic form $Q$, whose corresponding bilinear
form  has Gram matrix $J$ with respect to the canonical basis.
Imposing that $Q$ is preserved by $y$ we get $Q(e_1)=0$ and $Q(e_2)=1$.
Next, imposing that $Q$ is preserved by $x$, we get $Q(e_1)=Q(e_2)$, whence
the contradiction $0=Q(e_1)=Q(e_2)=1$.

This proves that $H=\Sp_{12}(q)$. The last part of the statement follows from Lemma \ref{q4}.
\end{proof}

\section{Case $n=8$ and $q>2$}\label{caso n=8}

Although computational evidences suggest that the elements $x,y$ described in Section \ref{sec2}
work also for $n=8$, we prefer for sake of simplicity to take generators similar to those used for $n=5$.
So, define $x,y$ of respective order $2$ and $3$, as 
follows:
\begin{itemize}
\item $x$ swaps $e_{\pm 1}$ with $e_{\pm 2}$, swaps $e_{\pm 4}$ with $e_{\pm 5}$, and fixes $e_{\pm 3}$;
\item $x$ acts on $\langle e_6,e_7,e_8\rangle$ with matrix
$$\zeta=\begin{pmatrix} -1 & 0 & a \\ 0 & -1 & 0 \\ 0 & 0 & 1 \end{pmatrix}, \quad a \in \F_q^*,$$
and acts on $\langle e_{-6},e_{-7}, e_{-8} \rangle$ with matrix $\zeta^T$;
\item $y e_i=e_{-i}$ and $ye_{-i}=-(e_{i}+e_{-i})$ for $i=1,8$;
\item  $y$ acts on $\langle e_{\pm 2}, \ldots, e_{\pm 7}\rangle$ as the permutation
$\left(e_{\pm 2 },e_{\pm 3},e_{\pm 4}\right)\left(e_{\pm 5 },e_{\pm 6},e_{\pm 7}\right)$.
\end{itemize}
Since $x^T J x  = y ^T J y = J$, it follows that $H=\langle x,y\rangle$ is contained in $\Sp_{16}(q)$.
Furthermore, by Proposition \ref{q=2} we may assume $q>2$.
We start proving the absolute irreducibility of $H$.

\begin{lemma}\label{irr8}
If $p=2$, the subgroup $H$ is absolutely irreducible. 
\end{lemma}

\begin{proof}
Take the element $\eta=y^2[x,y]^3y^2$: its characteristic polynomial is
$(t + 1)^4(t^2 +t + 1)^6$.
The eigenspaces of $\eta$ and $\eta^T$ relative to the eigenvalue $1$ are 
unidimensional and generated, respectively, by
$s_{1}=e_2+e_5+e_7$ and $\overline s_{1}=e_{-2}+e_{-5}+ e_{-7}$.
Let $U$ be an $H$-invariant subspace of $\F^{16}$ and $\overline{U}$ be a corresponding $H^T$-invariant  complement.\\
\textbf{Case 1.} The restriction $\eta_{|U}$ has the eigenvalue $1$.
Then $s_{1}$ belongs to $U$. 
Taking the vector
$u=s_{1}+ x s_{1}+ ys_{1}+ yx s_{1}+ y^2x s_{1}+ yxy(yx)^2s_{1} + (xy)^2  (y x)^2 s_{1}$, we obtain $u=a e_{-8}$.
So,  making alternating use of $x$ and $y$ it is easy to show that
$U$ contains the canonical basis $\{e_1,\ldots,e_{8},e_{-1},\ldots,e_{-8}\}$.
We conclude that $U=\F^{16}$.\\
\textbf{Case 2.} If Case 1 does not occur, then $\overline U$ contains the eigenvector $\overline s_{1}$.
Taking $\overline u=x^T \overline s_{1}+ (yx)^T \overline s_{1} + (y^2)^T \overline s_{1} + 
((xy)^2x)^T \overline s_{1} + ((xy)^3)^T \overline s_{1} + ((xy)^5)^T \overline s_{1}$,
we get $\overline u = a e_{-8}$.
 It follows that $\overline{U}=\F^{16}$ and $U=\{0\}$.
\end{proof}

\begin{lemma}\label{irr8b} 
If $p$ is odd, the group $H$ is absolutely irreducible.
\end{lemma}

\begin{proof} 
The eigenspace $\V_{1}([x,y])$ is  generated by 
$w_1$ and  $w_2$, where
$$w_1= e_1+e_2-e_4-e_5+e_{-1}+e_{-2}+2a^{-1}e_{-8} \equad w_2=e_6.$$
Note that $xw_1=w_1$ and $xw_2=-w_2$.

Let $U$ be an $H$-invariant subspace of $\F^{16}$.
If  $U$ contains $w_2$,  making alternating use of $x$ and $y$ it is easy to show that
$U$ contains the canonical basis $\{e_1,\ldots,e_{8},e_{-1},\ldots,e_{-8}\}$.
We conclude that $U=\V$.
By Lemma \ref{HT}(2), we have $\V_{1}([x,y])\cap U=\langle u \rangle$,
where we may assume $\dim(U)\leq 8$ and, by the previous computations, $u=w_1$.
Let $M$ be the matrix whose columns are the images of $w_1$ under 
$$\I_{16}, \;y, \; y^2, \;  xy^2, \; yxy^2, \; [x,y^2], \; (yx)^2y^2, \; x(yx)^2y^2, \; y^2(xy^2)^3.$$
The submatrix of $M$ corresponding to the rows $5, 7, 10, 11, 12, 13, 14, 15, 16$
has determinant  $4a$.
We conclude  that $M$ has rank $9$, against our assumption $\dim(U)\leq 8$.
\end{proof}

Now, take $\tau=[x,y]^{4}$ if $p=2$ and $\tau=[x,y]^8$ if $p$ is odd.
Then, $\tau$ belongs to $N \rtimes \wSL_8(q)$ and its characteristic polynomial is
$(t-1)^{16}$. The subgroup
$$K=\left\{\begin{array}{ll}
\left\langle \tau^{yxy^2}, \tau^{y x y^2 x}, \tau^{(yx)^3}, \tau^{(yx)^2y } \right\rangle & \textrm{if } p=2,\\[5pt]
\left\langle  \tau^{yx y^i}, \tau^{(yx)^2}, \tau^{(yx)^2y}, \tau^{(yx)^3}, \tau^{y^2xy^2}\mid i\in \{0,\pm 1\} \right\rangle 
& \textrm{if } p>2
       \end{array}\right.$$
fixes $V$, namely $K$ is a subgroup of $N \rtimes \wSL_8(q)$.

We first consider the case $p=2$.
Take the matrix 

\begin{footnotesize}
$$P=\begin{pmatrix}
0 & 0 & 1 &    0 & 1 & 1 & 0 & 0\\
0 & 0 & 1 &  a^2 & 0 & 0 & a^2 & 0\\
0 & 0 & 0 &    0 & 1 & 0 & a^2 & 0\\
1 & 0 & 0 &    0 & 0 & 0 & 0 & 0\\
0 & 1 & 0 &    0 & 0 & 0 & 0 & 0\\
0 & 0 & 0 &    0 & 1 & 0 & 0 & 0\\
0 & 0 & 1 &    0 & 0 & 0 & 0 & 0\\
0 & 0 & 0 &    0 & 0 & a^{-1} & a & 1
    \end{pmatrix}$$ 
\end{footnotesize}

\noindent whose determinant is $a^4$. Define now $\hat P=\diag(P,P^{-T})$:
then $K^{\hat P}\leq N \rtimes \wSL_8(q)$  acts as the identity on the subspace $E_3=\langle e_1,e_2,e_3\rangle$.
Let $\hat K$ be the projection on $V/E_3$ of $K^{\hat P}$.
The group $\hat K$ is then generated by the following elements:
$$
\tau_1=
\begin{pmatrix}
 1 &  1 &  0 &  0 &  0\\
 0 &  1 &  0 &  0 &  0\\
 0 &  0 &  1 & a^4&   0\\
 0 &  0 &  0 &  1 &  0\\
 0 &  0 &  0 &   0 &  1\\
\end{pmatrix},\quad 
\tau_2=
\begin{pmatrix}
 1 &   0 &  0  &  0 &  0\\
 0 &  1  & a^2 &  0  & a^3\\
 0 & a^2 &  1 & a^4 &  0\\
 0 &  0 &  1 &  1  & a\\
 0 &  a &  0  & a^3 &  1\\
\end{pmatrix},$$
$$\tau_3=\begin{pmatrix}
      1  &     0  &     0    &   0 &      0\\
    a^4 & a^2 + 1  &   a^2    &   0 &    a^3\\
      0  &   a^2 &a^2 + 1    & a^4 &    a^3\\
      0  &     0 &      0    &   1 &      0\\
    a^3  &     0 &      0    & a^3 &      1\\
       \end{pmatrix},\quad
\tau_4=\begin{pmatrix}
1& 1& 0& 0& 0\\
0& 1& 0& 0& 0\\
0& 0& 1& 0& 0\\
0& 0& 1& 1& 0\\
0& 0& 0& 0& 1\\
       \end{pmatrix}. $$

\begin{lemma}\label{8p2}
Assume $p=2 < q$. If $a\in \F_q^*$ is such that $\F_2[a]=\F_q$, then $\hat K =\SL_{5}(q)$.
\end{lemma}

\begin{proof}
We first prove that $\hat K$ is absolutely irreducible.
To this purpose, we observe that the characteristic polynomial of $\tau_1 \tau_4$ is $(t + 1)^3 (t^2 + a^4 t + 1)$.
Let $\sigma\neq 1$ be an eigenvalue of $\tau_1 \tau_4$.
Then the eigenspaces of $\tau_1\tau_4$ and $(\tau_1 \tau_4)^T$ relative to the eigenvalue $\sigma^{\pm 1}$ are generated, respectively, by
$s_{\sigma^{\pm 1}}= (\sigma^{\pm 1} + 1) e_3+e_4$ and $\overline s_{\sigma^{\pm 1}}= e_3+(\sigma^{\mp 1} + 1)e_4$.
Let $U$ be a $\hat K$-invariant subspace of $\F^{5}$ and $\overline{U}$ be a corresponding $\hat K^T$-invariant  complement.\\
\textbf{Case 1.} The restriction $(\tau_1\tau_4)_{|U}$ has an eigenvalue $\sigma^{\pm 1}$.
The matrix whose columns are the five vectors 
$$s_{\sigma^{\pm 1}}, \; \tau_1 s_{\sigma^{\pm 1}}, \; \tau_2 s_{\sigma^{\pm 1}},\;
 \tau_3\tau_1 s_{\sigma^{\pm 1}}, \; \tau_4 \tau_2 s_{\sigma^{\pm 1}}$$
has determinant $a^{15}(\sigma^{\pm 1}+1)\neq 0$, so $U=\F^5$.\\
\textbf{Case 2.} If Case 1 does not occur, then $\overline U$ contains both the eigenvectors $\overline s_{\sigma}$ and $\overline s_{\sigma^{-1}}$.
Take $\overline u =\overline s_{\sigma}+ \overline s_{\sigma^{-1}}$. The matrix whose columns are the five vectors
$$\overline u=a^4e_4, \; \tau_2^T\overline u,\;\tau_4^T \overline u, \; (\tau_2\tau_3)^T \overline u,\; (\tau_4\tau_3)^T\overline u$$
has determinant $a^{27}$. We conclude that $\overline U=\F^{5}$ and $U=\{0\}$.
 
Suppose now that $\hat K$ is imprimitive. Then it permutes transitively $5$
appropriate subspaces $\langle v_i\rangle$, $1\leq i \leq 5$.
In particular $\tau_1\tau_4$ must have a non-trivial orbit in this permutation action. 
Since this element fixes pointwise a $3$-dimensional space, the lengths of the orbits can only be $(2,1^3)$, or $(3,1^2)$, or
$(2^2,1)$, with corresponding characteristic polynomials
$$(t^2+\alpha)(t+1)^3,\quad (t^3+\alpha)(t+1)^2\equad (t^2+\alpha)(t^2+\beta)(t+1).$$
Comparing these polynomials with the characteristic polynomial of $\tau_1\tau_4$, each alternative leads to the contradiction 
$a\in \F_2$.

By \cite[Tables 8.18 and 8.19]{Ho}: $\hat K$ is a classical group in a natural representation.
Suppose there exists $g\in \GL_{5}(\F)$ such that $\hat K^g\leq \GL_{5}(q_0)(\F^* \, \I_{5})$, where $\F_{q_0}$ is 
a subfield of $\F$. Set $(\tau_1\tau_4)^g=\vartheta^{-1} \tau_0$, with $\tau_0\in \GL_5 (q_0)$ and  $\vartheta\in \F$.
Since $\tau_1\tau_4$ has the similarity invariant $t+1$, the matrix $\tau_0=\vartheta(\tau_1\tau_4)^g$ has the similarity 
invariant  $t+\vartheta$. It follows that $\vartheta \in \F_{q_0}$ and that the trace of $\tau_1\tau_4$ belongs to the subfield
$\F_{q_0}$. From  $\tr (\tau_1\tau_4)=a^4 + 1$ we obtain $a\in \F_{q_0}$, and so
$\F_q=\F_{2}[a]\leq \F_{q_0}$. We conclude that either $K=\SL_{5}(q)$ or 
$K\leq \SU_{5}(q_1^2) (\F_q^* \, \I_{5})$ with $q_1^2=q$.
Suppose the latest case occurs. 
Let $\psi: \GL_{5}(q)\to \GL_{5}(q)$ be the morphism defined by $(a_{i,j})^\psi=(a_{i,j}^{q_1})$. 
Taking $g=[\tau_1, \tau_4]$ and using the fact that $g^{-T}$ must be conjugate to $g^{\psi}$,
we obtain that $\tr(g) \in \F_{q_1}$. 
Since $\tr(g)=(a+1)^{8}$, we obtain that $a \in \F_{q_1}$ and so $\F_q=\F_2[a]\leq \F_{q_1}$.
This proves that $\hat K=\SL_{5}(q)$.
\end{proof}

Assume now $p$ odd. Observe that in this case $\tau$ is a bireflection.
Moreover, the subgroup $K$ preserves the subspace $E_5=\langle e_1,\ldots,e_5\rangle$, acting as the identity on $V/E_5$.
Let $\tau_1,\ldots,\tau_7$ be the restrictions to $E_5$ of the seven generators of $K$. Then
$$\tau_1=\I_5+4 a^2 (E_{4,5}-E_{4,2}),\quad \tau_2=\I_5-4a^2 E_{3,4},\quad 
\tau_3=\I_5 - 4a^2 E_{2,3},$$ 
$$\tau_4=\I_5-4a^2 E_{3,5},\quad 
\tau_5=\I_5+4a^2 E_{2,1},\quad  \tau_6=\I_5+4a^2 E_{1,2},\quad \tau_7=\I_5+4a^2( E_{5,1}-E_{5,4}).$$

\begin{lemma}\label{8podd}
Assume $p$ odd and $q\neq 9$. Suppose that $a\in \F_q^*$ is such that:
$${\rm (a)}\quad \F_p[a^4]=\F_q;\qquad  {\rm (b)}\quad 4a^4+1\neq 0.$$
Then $K_{|E_5} =\SL_{5}(q)$.
\end{lemma}

\begin{proof}
By Dickson's Lemma and  assumption (a), 
the group $\left\langle \tau_5 , \tau_6\right\rangle$ induces $\SL_2(q)$ on $\langle e_1,e_2\rangle$. 
We now prove that $K_{|E_5}$ is absolutely irreducible.
To this purpose, we observe that $\tau_5\tau_6$ is a symmetric matrix whose characteristic polynomial 
is $(t -1)^3 (t^2 -2(8a^4+1)  t + 1)$.
Let $\sigma\neq 1$ be an eigenvalue of $\tau_5 \tau_6$ (note that $-1\neq \sigma\neq \sigma^{-1}$ by (b)).
The eigenspace of $\tau_5\tau_6$ relative to the eigenvalue $\sigma^{\pm 1}$ is generated by 
$s_{\sigma^{\pm 1}}= 4e_1 + \frac{\sigma^{\pm 1}-1}{a^2} e_2$.
Let $U$ be a $K_{|E_5}$-invariant subspace of $\F^{5}$ and $\overline{U}$ be a corresponding $(K_{|E_5})^T$-invariant  complement.\\
\textbf{Case 1.} The restriction $(\tau_1\tau_2)_{|U}$ has an eigenvalue $\sigma^{\pm 1}$.
Since the matrix whose columns are the five vectors
$$s_{\sigma^{\pm 1}},\; \tau_6\tau_1s_{\sigma^{\pm 1}},\;
\tau_6\tau_2 s_{\sigma^{\pm 1}},\;
\tau_6\tau_7 s_{\sigma^{\pm 1}},\;
\tau_6\tau_2\tau_1 s_{\sigma^{\pm 1}}$$
has determinant $-2^{12} a^2 (\sigma^{\pm 1}-1)^4\neq 0$, we obtain $U=\F^5$.\\
\textbf{Case 2.} If Case 1 does not occur, then $\overline U$ contains both the eigenvectors
$s_{\sigma}$ and $s_{\sigma^{-1}}$. Take $\overline u=s_{\sigma}-s_{\sigma^{-1}}$.
Since the matrix whose columns are the five vectors
$$\overline u, \; (\tau_3)^T\overline u, \; (\tau_5)^T\overline u, \;
(\tau_6\tau_3\tau_4)^T\overline u, \; (\tau_6\tau_3 \tau_4 \tau_2)^T\overline u$$
has determinant $2^{12} a^2 (\sigma-\sigma^{-1})^5\neq 0$, we obtain $\overline U=\F^5$ and so $U=\{0\}$.

By Lemma \ref{McLaughlin} we conclude that   $K_{|E_5}=\SL_5(q)$.
\end{proof}

\begin{proposition}\label{main8}
Suppose $q>2$. Then there exists $a \in \F_q^*$ such that $H=\Sp_{16}(q)$. 
\end{proposition}

\begin{proof}
We first show that for $q\not \in \{2,5,9\}$, there exists $a \in \F_q^*$ such that 
$\F_p[a^4]=\F_q$ and $4a^4+1\neq 0$.
If $p=2$, it suffices to take a  primitive element of $\F_q$.
If $q=p\geq 7$, we have to exclude at most $4$ values of $a$,
so we are done.
If $q=3$ take $a\neq 0$.
Suppose now  $q=p^f$, with $f\geq 2$, and let $\N(q)$ be the number of elements $b\in \F_q^*$ 
such that $\F_p[b^4]\neq \F_q$. We have to check when $p^f-1-\N(q) > 0$,
as the condition  $a^4\neq -\frac{1}{4}$ can be dropped.
By Lemma \ref{campoN}, we have
$\N(q)\leq 4 p\frac{ p^{\left \lfloor f/2\right \rfloor} -1}{p-1}$
and hence it suffices to check when
$p^f - 4 p\frac{ p^{\left \lfloor f/2\right \rfloor} -1}{p-1} > 1$.
This condition is fulfilled, unless $q=3^2$.

So, taking a such element $a\in \F_q^*$, by Lemmas \ref{irr8}, \ref{8p2} and \ref{8podd}
the subgroup $H$ is absolutely irreducible and its order is divisible by the order of $\SL_5(q)$.
Suppose that $H$ is contained in a maximal subgroup $M$ of $\Sp_{16}(q)$.

Suppose that $M\in \mathcal{C}_5$. Then, for every $h \in H$  we have $h^2\in \Sp_{16}(\F_{q_0})$ for some subfield
$\F_{q_0}$ of $\F_{q}$. In particular, if $p=2$ we get
$a^4  =\tr ((xy)^{18}) \in \F_{q_0}$;
if $p$ is odd, we get
$8a^2 - 1=\tr ((xy)^{8}) \in \F_{q_0}$. 
In both cases, we conclude $q_0=q$, by the assumption $\F_q= \F_p[a^4]$.

Hence, by order reasons, we have that $M$ can only be one of the following subgroups:
(i) a subgroup of type $\GL_8(q).2$  in class $\Cl_2$ ($q$ odd);
(ii) a subgroup in class $\Cl_6$ ($q=p>2$);
(iii) a subgroup in class $\Cl_8$ ($q$ even);
(iv) a subgroup in class $\ClS$.

Suppose $p$ odd. If $H$ preserves a decomposition  $V=V_1\oplus V_2$, where $\dim V_1=\dim V_2=8$, 
by the irreducibility, the permutation action must be transitive: we have $yV_1=V_1$,
$yV_2=V_2$ and $xV_1=V_2$.  This produces the contradiction $-1=\tr (xy)=0$.
Since $\tau$ is a bireflection, we can now apply \cite[Theorem 7.1]{GS}: $K$ is a classical group in a natural representation
and by the above considerations, we conclude that $H=\Sp_{16}(q)$.

Suppose $p=2$. If $M$ belongs to the class $\Cl_8$, $H$  preserves a quadratic form $Q$, whose corresponding 
bilinear form  has Gram matrix $J$ with respect to the canonical basis.
Imposing that $Q$ is preserved by $y$ we obtain $Q(e_1)=1$, $Q(e_2)=Q(e_4)$,
and $Q(e_5)=Q(e_6)$.  Imposing that $Q$ is preserved by $x$, we get $Q(e_1)=Q(e_2)$, $Q(e_4)=Q(e_5)$
and $Q(e_6)=0$. This produces the contradiction $1=Q(e_1) =Q(e_6)  =0$.
So, we are left to the case $M\in \ClS$. 
The quasi-simple groups which admit irreducible representations of degree $16$ are listed in 
\cite{HM}. Since the order of $H$ is divisible by the order
of $\SL_5(q)$, we can excluded all the possibilities.
We conclude that $H=\Sp_{16}(q)$.

Finally, consider the cases $q\in  \{5,9\}$.
If $q=5$, take $a=1$; if $q=9$, take $a\in \F_9$ whose minimal polynomial over $\F_3$ is $t^2-t-1$.
Also, take
$L_5=\{3,9,10,13,14,15,34 \}$ and
$L_9=\{4,6,7,8,11,15,20,54 \}$.
Then 
$$\bigcup_{k \in L_q} \pi( (xy)^k y ) =\pi(\Sp_{16}(q)).$$
By Lemma \ref{LPS} we conclude that  $H=\Sp_{16}(q)$.
\end{proof}

\end{document}